\documentclass[a4paper,tbtags,leqno,twoside,12]{article}%

\usepackage{amsmath}
\usepackage{amsthm}
\usepackage{amssymb}

\theoremstyle{plain}
\newtheorem{thm}{Theorem}[section]  

\newtheorem{lem}[thm]{Lemma}
\newtheorem{prop}[thm]{Proposition}
\theoremstyle{definition}
\newtheorem{dfn}{Definition}[section]

\theoremstyle{remark}
\newtheorem{rem}{Remark}[section]

\newcommand\DN{\newcommand}

\DN\lref[1]{Lemma~\ref{#1}}
\DN\tref[1]{Theorem~\ref{#1}}
\DN\pref[1]{Proposition~\ref{#1}}
\DN\sref[1]{Section~\ref{#1}}
\DN\dref[1]{Definition~\ref{#1}}
\DN\rref[1]{Remark~\ref{#1}} 
\DN\correff[1]{Corollary~\ref{#1}}
\DN\eref[1]{Example~\ref{#1}}

\DN\map[3]{#1\!:\!#2\!\to\!#3}
\DN\limi[1]{\lim_{#1\to\infty}} 	
\DN\limz[1]{\lim_{#1\to 0}} 	
\DN\lsupi[1]{\limsup_{#1\to\infty}} 	
\DN\linfi[1]{\liminf_{#1\to\infty}} 	
\DN\lsupz[1]{\limsup_{#1\to 0}} 	
\DN\linfz[1]{\liminf_{#1\to 0}} 	
\DN\PD[2]{\frac{\partial#1}{\partial#2}}

\numberwithin{equation}{section}
\newcounter{Const} 

\def\Ct{\refstepcounter{Const}c_{\theConst}}
\def\cref#1{c_{\ref{#1}}}

\title{Interacting Brownian motions \\ in infinite dimensions \\ with 
7logarithmic interaction potentials II: \\ Airy random point field }

\author{Hirofumi Osada 
\\Faculty of Mathematics, Kyushu University, \\ 
Fukuoka 819-0395, Japan
} 
\begin{document}

\maketitle

\footnote{

\noindent \\
{\bf E-mail:} osada@math.kyushu-u.ac.jp  \quad 
{\bf Phone and Fax:} 0081-92-802-4489}
\begin{abstract}
We give a new sufficient condition of the quasi-Gibbs property. 
This result is a refinement of one given in a previous paper (\cite{o.rm}), 
and will be used in a forth coming paper to prove 
the quasi-Gibbs property of Airy random point fields (RPFs) and 
other RPFs appearing under soft-edge scaling. 
The quasi-Gibbs property of RPFs is one of the key ingredients 
to solve the associated infinite-dimensional stochastic differential equation (ISDE). 
Because of the divergence of the free potentials and 
the interactions of the finite particle approximation under soft-edge scaling, 
the result of the previous paper excludes the Airy RPFs, 
although Airy RPFs are the most significant RPFs 
appearing in random matrix theory. 
We will use the result of the present paper to solve the ISDE 
for which the unlabeled equilibrium state is the $\mathrm{Airy}_{\beta }$  RPF with $ \beta = 1,2,4 $.  \end{abstract}

\DN\Aq{\mathsf{A}_{rs,l}^{n,q}}

\DN\Srm{\SSS _{r}^{m}}
\DN\Srhat{\tilde{ S }_r}
\DN\SSrhat{\tilde{\SSS }_r}
\DN\Srhatn{\tilde{ S }_r^n}
\DN\SSrhatn{\tilde{\SSS }_r^n}

\DN\Sr{ S _{r}} 
\DN\Ss{ S _{s}} 
\DN\Srs{ S _{rs}} 
\DN\Srinfty{ S _{r\infty}} 
\DN\Ssinfty{ S _{s\infty}} 
\DN\Sti{\tilde{ S }}
\DN\Sri{ S _{r\infty}}
\DN\Ssi{ S _{s\infty}}
\DN\Stri{\Sti _{r\infty}}
\DN\Dr{\Sti _r} 
\DN\Ds{\Sti _s} 
\DN\Drs{\Sti _{rs}} 
 \DN\Rr{\Sti _{1r}} 
 \DN\Rrr{ S _{r(r+1)}}  
 \DN\Rrrr{ S _{(r+1)\infty }}

\DN\n{\mathsf{n}}
\DN\nN{N_{\mathsf{n}}}

  \DN\4{\sum_{\y _i \in \Srs }}
\DN\9{\sum_{\y _i \in \Srs } \frac{1}{{ y }_i^{\ell } }}

\DN\vellrs{\mathsf{v}_{\ell ,rs}}
\DN\vellri{\mathsf{v}_{\ell ,r\infty }}
\DN\vellsi{\mathsf{v}_{\ell ,s\infty }}
\DN\vellbr{\mathsf{v}_{\ell ,\br }}
\DN\vellbs{\mathsf{v}_{\ell ,\bs }}
\DN\vellbri{\mathsf{v}_{\ell ,\br \infty }}
\DN\vellbsi{\mathsf{v}_{\ell ,\bs \infty }}

\DN\DeltaN{\Delta ^{\n ,n}}

\DN\supN{\sup_{\n \in \mathbb{N}   }}
\DN\supP{\sup_{\mathsf{p} \in \mathbb{N}   }}
\DN\infN{\inf_{\n \in \mathbb{N}   }}
\DN\sups{\sup_{r< s\in  \mathbb{N}   }}
\DN\lsupN{\limsup_{\n \to\infty}}
\DN\linfN{\liminf_{\n \to\infty}}
\DN\capN{\bigcap_{\n \in \mathbb{N}   }}
\DN\cupN{\bigcup_{\n \in \mathbb{N}   }}

\DN\limsupN{\limsup_{\n \to \infty }}
\DN\sumN[1]{\sum_{#1 = 0}^{\n -1}}

\DN\x{x} 
\DN\y{y}
\DN\xx{\mathsf{x}} 

\DN\p{\mathsf{p}}
\DN\q{\mathsf{q}}
%%%%%%%%%%%%%%%%%%%%%%%%%%%%%%%%%%%%%%%%%%%%%%%%

\DN\HH{\mathsf{H}}
\DN\Hrk{\HH _{r,k}}
\DN\Hsl{\HH _{s,l}}

\DN\PhiN{\Phi ^{\n } }

\DN\PsiN{\Psi ^{\n }}   
\DN\PsiNr{\PsiN _{r}} %A
\DN\PsiNrs{\PsiN _{r,rs}} %AB
\DN\PsiNrstu{\PsiN _{rs,tu}} %AB
\DN\PsiNrsst{\PsiN _{rs,st}} 
\DN\PsiNrst{\PsiN _{r,st}} %AB
\DN\PsiNrsi{\PsiN _{r,s\infty }} %AC

\DN\wPsiN{\widetilde{\Psi }^{\n }}   
\DN\wPsiNr{\wPsiN _{r}} %A
\DN\wPsiNrs{\wPsiN _{r,rs}} %AB
\DN\wPsiNrstu{\wPsiN _{rs,tu}} %AB
\DN\wPsiNrst{\wPsiN _{r,st}} %AB
\DN\wPsiNrsi{\wPsiN _{r,s\infty }} %AC
\DN\wPsiNssi{\wPsiN _{rs,s\infty }} %BC
\DN\PsiNssi{\PsiN _{rs,s\infty }} %BC
\DN\wPsiNsr{\wPsiN _{rs* }} %BC
\DN\wPsiNrsst{\wPsiN _{rs,st}}

\DN\ABN{_{r,k,\mathsf{s},rs}^{\n ,m}}
\DN\ABs{_{r,k,\mathsf{s},rs}^{m}}
\DN\ABss{_{r,k,\pi _{\Srs }(\mathsf{s}),rs}^{m}}

   \DN\mukNm{\mu ^{\n ,m}_{r,k}} %IKI 
   \DN\murk{\mu _{r,k}^{m}} %生き
   \DN\murky{\mu _{r,k,\mathsf{s}}^m} %これは生き
   \DN\mury{\mu _{r,\mathsf{s}}^m} %これは生き
   \DN\muN{\mu ^{\n }} %IKI 
\DN\muNl{\mu^{\n ,n}_{l}}
   \DN\muA{\mu _{r,k,\mathsf{s}}^{ m }} %IKI 
\DN\muABs{\mu \ABs }%IKI 
\DN\muABN{\mu \ABN }%IKI 

\DN\murmy{\mu _{r,k,\mathsf{s}}^m } %これは生き
\DN\murm{\mu _{r}^m } %IKI 
\DN\murmch{\check{\mu } _{r,k,\mathsf{s}}^m}%これは生き
\DN\murmcha{\check{\mu } _{r,k,\mathsf{s},a_0 }^m}%これは生き

\DN\rN{\rho _{\n }^{n}} %IKI
\DN\rNone{\rho _{\n }^{1}} %IKI 
\DN\LAB{L^{\infty }(\SSS , \Lambda )}
\DN\LABone{L^{1}(\SSS , \Lambda )}
\DN\skABs{\sigma \ABs }
\DN\skABss{\sigma \ABss }
\DN\st{\widetilde{\sigma}} %IKI 

\DN\sAym{\sigma _{r,k,\mathsf{s}}^{ m }} %%iki %%
\DN\sABN{\sigma \ABN }

\DN\ab{r,rs}
\DN\tC{\tau _{\ab }^{\n }}
\DN\tCNtilde{\widetilde{\tau }_{\ab }^{\n } (\mathsf{x},\mathsf{s})}
\DN\tCN{\tau _{\ab }^{\n } (\mathsf{x},\mathsf{s})}
\DN\tCNNtilde{\widetilde{\tau }_{\ab }^{\n } (\mathsf{x}',\mathsf{s})}
\DN\tCNN{\tau _{\ab }^{\n } (\mathsf{x}',\mathsf{s})}

\DN\piAc{\pi _{\SrNstar }}

%\DR\S{S}       
\DN\SSS{\mathsf{S}}
\DN\Lm{L^2(\SSS , \mu )}
\DN\Ln{L^2(\SSS , \nu )}
\DN\LmN{L^2(\SSS , \muN )}
\DN\LmNone{L^1(\SSS , \muN )}
\DN\Lmk{L^2(\SSS , \murk )}

\DN\ak{k} 
\DN\br{b_r }\DN\bs{b_s }\DN\bt{b_t }
\DN\brr{b_{r+1}}
\DN\bsr{b_s }\DN\btr{b_t }
\DN\brrr{r }\DN\bsrr{s }\DN\btrr{t }

\DN\mA{\Lambda (d\mathsf{x})} % この2つは同じ
\DN\mAN{\Lambda (d\mathsf{x})}

\DN\kk{\ell }
\DN\ellell{{\ell _0}}
\DN\kK{ 1 \le \ell < \ellell }

%%%%%%%%% [] new [] %%%%%%%%% 新たに
\DN\mmi{\mathfrak{m}^{\n }_{\infty}}
\DN\mms{\mathfrak{m}^{\n }_{s}}
\DN\mmr{\mathfrak{m}^{\n }_{r}}
\DN\mmt{\mathfrak{m}^{\n }_{t}}
\DN\mmu{\mathfrak{m}^{\n }_{u}}
\DN\mmmm{F^{\n }_{rs}(\mathsf{y})} 
\DN\mmm{\beta \big\{ \4 \frac{1}{ x _i } \big\} + (\mmr -\mms ) } 

\DN\xr{\sum_{x_i\in\Sr } x_i }
\DN\xrdash{\sum_{x'_i\in\Sr } x'_i }
\DN\xrs{\sum_{x_i\in\Srs } x_i }
\DN\xsi{\sum_{x_i\in\Ssi } x_i }

\DN\mairybeta{m _{\mathrm{Ai},\beta }}
\DN\muairybeta{\mu _{\mathrm{Ai},\beta }}

%%%%%%% definition end []

\section{Introduction} \label{s:1} 
Let $ \beta = 1,2,4$. 
The $\mathrm{Airy}_{\beta }$  random point field (RPF), denoted by 
$ \muairybeta $, is a probability measure 
on the configuration space over $ \mathbb{R} $, for which the 
$ n$-correlation function $ \rho _{\mathrm{Ai } ,2}^{n}$ is given by 
\begin{align}\label{:a11}&
\rho _{\mathrm{Ai } ,2}^{n} (x_1,\ldots,x_n) 
= \det  [K_{\mathrm{Ai } , 2}(x_i,x_j) ]_{i,j=1}^{n} 
\quad \text{ for }\beta = 2
.\end{align}
Here $ K_{{\mathrm{Ai } },2}(x,y)$ is a continuous kernel on $ \mathbb{R}^2$ defined by 
\begin{align*}&
K_{{\mathrm{Ai } },2}(x,y) =  
\frac{{\mathrm{Ai } }(x) {\mathrm{Ai } }'(y)-{\mathrm{Ai } }'(x) {\mathrm{Ai } }(y)}{x-y} \quad (x\not=y) 
,\end{align*}
where we set ${\mathrm{Ai } }'(x)=d {\mathrm{Ai } }(x)/dx$ with ${\mathrm{Ai } }(\cdot)$ 
denoting the Airy function 
\begin{align}\label{:11p}&
{\mathrm{Ai } }(z) = \frac{1}{2 \pi} \int_{\mathbb{R}} dk \,
e^{i(z k+k^3/3)},
\quad z\in \mathbb{C} 
.\end{align}
The correlation functions of $\mathrm{Airy}_{\beta }$ RPFs for $ \beta = 1,4 $ 
are given similarly by using the quaternion determinant or Pfaffians 
(see \cite{agz}, \cite{mehta}, \cite{for}). 
 
It is well known that $ \muairybeta  $ results in the thermodynamic limit 
of the distributions for the Gaussian ensembles ($ \beta = 1,2,4 $). 
Indeed, the distribution of eigenvalues of the Gaussian ensembles 
with size $\n \times \n $ is given by
\begin{equation}
m_{\mathrm{Gauss},\beta }^{\n }(d\mathbf{\x }_{\n })= 
\frac{1}{Z}\prod_{i<j}^{\n }|x_i-x_j|^\beta 
\exp\left\{-\frac{\beta}{4}\sum_{i=1}^{\n } |x_i|^2 \right\}d\mathbf{\x }_{\n },
\label{:11z}
\end{equation}
where $ \mathbf{\x }_{\n } = (x_1,\ldots,x_{\n }) \in \mathbb{R}^{\n }$. 
Here $\beta=1, 2$ and $4$ correspond respectively to the Gaussian 
orthogonal (GOE), 
unitary (GUE), and symplectic (GSE) ensembles. 
Thus, the probability density coincides with the Boltzmann factor 
for  log-gas systems at three special values of the inverse temperature, i.e., 
$\beta=1, 2$ and $4$. 

Let $ \mu _{\mathrm{Gauss},\beta }^{\n }  $ 
be the distribution of $ \n ^{-1} \sum \delta_{x_i} $ 
under $ m_{\mathrm{Gauss},\beta }^{\n }(d\mathbf{\x }_{\n }) $. 
Then the celebrated semi-circle law states that 
$ \mu _{\mathrm{Gauss},\beta }^{\n } $ converge to the nonrandom 
$ \sigma (x) dx $ weakly in the space of Radon measures  
over $ \mathbb{R } $ endowed with the vague topology. Here 
\begin{align}\label{:11a}&
\sigma (x) =  \frac{1}{2\pi }1_{[-2,2] } (x)  \sqrt{4 - x^2 }
.\end{align}

There exist two typical thermodynamic scalings in \eqref{:11z}, called 
bulk and soft-edge. 
The former (centered at the origin) is given by the correspondence 
$ x \mapsto x/\sqrt{\n }$, which yields the RPF $ \mu^{\n }_{\mathrm{bulk},\beta } $ 
with labeled density $ m_{\mathrm{bulk},\beta }^{\n } $ such that 
\begin{align}\label{:11b}& 
m_{\mathrm{bulk},\beta }^{\n }(d\mathbf{\x }_{\n })= 
\frac{1}{Z}\prod_{i<j}^{\n }|x_i-x_j|^\beta 
\exp\left\{-\frac{\beta}{4\n }\sum_{i=1}^{\n } |x_i|^2 \right\}d\mathbf{\x }_{\n },
\end{align}
and  $ \mu^{\n }_{\mathrm{bulk},\beta } $ converges weakly to 
 $ \mu _{\mathrm{bulk},\beta } $, the $ \text{Sine}_{\beta }$ RPF. 
The latter, in contrast, is centered at $ 2\sqrt{\n }$ given by the correspondence 
$ x \mapsto 2\sqrt{\n } + x \n ^{-1/6}$ with labeled density 
$ \mairybeta ^{\n } $ such that 
\begin{align}\label{:11c}&
\mairybeta ^{\n }(d\mathbf{\x }_{\n })=\frac{1}{Z}\prod_{i<j}^{\n }   |x_i-x_j|^\beta
\exp\bigg\{-\frac{\beta}{4}\sum_{i=1}^{\n } |2\sqrt{\n }+\n ^{-1/6}x_i|^2 \bigg\}
.\end{align}
The Airy RPF $ \muairybeta $ is the weak limit of 
 $ \muairybeta ^{\n }$ given by $ \mairybeta ^{\n }$ as $ \n \to \infty $. 
The finite particle approximation $ \{  \muairybeta ^{\n } \} $ will be used 
in a forth-coming paper to prove the quasi-Gibbs property for $ \muairybeta $.

\bigskip

Interacting Brownian motions (IBMs) in infinite dimensions 
 are diffusions $ \mathbf{X}_t = (X_t^i)_{i\in \mathbb{Z}} $ consisting of infinitely many particles moving in $  \mathbb{R}^{d}  $ with the effect of the external force coming from a self-potential $ \map{\Phi }{ \mathbb{R}^{d}  }{ \mathbb{R} \cup \{ \infty \}}$ and that of the mutual interaction coming from an interacting potential 
 $ \map{\Psi }{ \mathbb{R}^{d}  \! \times \!  \mathbb{R}^{d}  }{ \mathbb{R} \cup \{ \infty \}} $ such that $ \Psi (x,y) = \Psi (y,x)$. 

Roughly speaking, an IBM is the stochastic dynamics of infinitely many  particles 
described by the infinite-dimensional stochastic differential equation (ISDE) 
of the form 
\begin{align} \label{:11} 
& dX_t^i = dB_t^i - \frac{1}{2} \nabla \Phi (X_t^i) dt 
-\frac{1}{2} 
\sum _{j \in  \mathbb{Z} , j\not = i } \nabla \Psi (X_t^i , X_t^j ) dt \quad 
(i \in  \mathbb{Z} )
.\end{align}
The state space of the process 
$\mathbf{X}_t = (X_t^i)_{i\in \mathbb{Z}}$ is $( \mathbb{R}^{d}  )^{\mathbb{Z}}$ 
by construction. 
Let $\mathsf{X}$ be the configuration-valued process given by 
\begin{align} \label{:12} 
& \mathsf{X}_t = \sum _{i\in  \mathbb{Z} } \delta _{X_t^i}
.\end{align}
Here $ \delta _a $ denotes the delta measure at $ a $ and 
a configuration is a Radon measure consisting of 
a sum of delta measures. 
We call $ \mathbf{X}$ the labeled dynamics and 
$ \mathsf{X}$ the unlabeled dynamics.

The ISDE \eqref{:11} was initiated by Lang \cite{La1}, \cite{La2}, 
who studied the case $ \Phi = 0 $, and 
$ \Psi (x,y) = \Psi (x-y) $, where $ \Psi $ 
is in $ C^3_0( \mathbb{R}^{d}  ) $, superstable and regular 
according to Ruelle \cite{ruelle2}. With the last two assumptions, 
the corresponding unlabeled dynamics $ \mathsf{X}$ has 
Gibbsian equilibrium states. 
See \cite{shiga}, \cite{Fr}, and \cite{T2} for other works 
concerning the SDE \eqref{:11}. 

In \cite{o.dfa}, the unlabeled diffusion was constructed using the Dirichlet form. 
The advantage of this method is that it gives a general and 
simple proof of construction.  
This work was followed by \cite{yoshida}, \cite{ark}, \cite{o.p}, \cite{o.m}, \cite{tane.udf}, \cite{yuu.05}, and others.  In all these, 
except \cite{yuu.05} and some parts of \cite{o.dfa}, 
the equilibrium states are supposed to be Gibbs measures 
with Ruelle's class interaction potentials $ \Psi $. 
Thus, the equilibrium states are described 
by the Dobrushin-Lanford-Ruelle (DLR) equations (see \eqref{:qg5}), 
the usage of which plays a pivotal role in previous works. 

The interaction potentials appearing in random matrix theory become 
logarithmic interaction potentials (2D Coulomb potentials): 
\begin{align} \label{:13} &
\Psi (x,y) = -\beta \log |x-y| , \quad 0 < \beta < \infty 
. \end{align}
Clearly these are not Ruelle's class potentials and 
the DLR equations would make no sense. 

In \cite{o.col}, \cite{o.tp}, \cite{o.rm} and \cite{o.isde}, 
we have developed a general theory 
applicable to log potentials and solved the ISDE \eqref{:11} 
with log interaction potentials. 
The key ingredients are two geometric properties of RPFs such that 
\lq\lq the quasi-Gibbs property'' and \lq\lq the log derivative''. 
Although we checked these for Sine$_{\beta} $ RPFs ($ \beta=1,2,4$) 
and the Ginibre RPF in \cite{o.rm} and \cite{o.isde}, 
the Airy$ _{\beta }$ RPFs remain.

The purpose of this paper is to give a sufficient condition 
for the quasi-Gibbs property applicable to RPFs 
appearing under soft-edge scaling, in particular, the  Airy$ _{\beta }$ RPFs. 
We will do this in the main theorems \tref{l:21} and \tref{l:22}. 

Let us briefly explain the main idea. 
The quasi-Gibbs property is a kind of existence of a locally bounded density 
conditioned outside (see Definition \ref{dfn:1}). We will prove this by uniform estimates of suitable, finite particle approximations. 
This finite particle system is \eqref{:11c} for the  Airy$ _{\beta }$ RPFs.  
Note that the exponent in \eqref{:11c}  is given by 
\begin{align}& \label{:14}
- \frac{\beta}{4}\sum_{i=1}^{\n } |2\sqrt{\n }+\n ^{-1/6}x_i|^2 = 
- \frac{\beta}{4}\sum_{i=1}^{\n } \{ 4{\n } +\n ^{-1/3}|x_i|^2 + 4\n ^{1/3}x_i \}
.\end{align}
The term $ 4{\n } $ can be absorbed in the normalizing constant, and 
the term $ \n ^{-1/3}|x_i|^2 $ can be neglected as $ \n \to \infty $. 
We have to prove, however, a rather precise cancellation between 
$ e^{-\frac{\beta}{4}\sum_{i=1}^{\n } 4\n ^{1/3}x_i }$ and the interaction term 
$ \prod_{i\not=j}^{\n } |x_i-x_j|^{\beta }$. This yields the main difficulty 
for the Airy$ _{\beta }$ RPFs, and other RPFs under soft-edge scaling. 
Note that the term $  4\n ^{1/3}x_i $ is linear in $ x_i $; from this, 
we arrive at the formulation in \eqref{:21o}.  

The organization of the paper is as follows. 
In \sref{s:2}, we describe the set up and state the main results 
(\tref{l:21}, \tref{l:22}). 
\sref{s:3}--\sref{s:5} are devoted to the proof of \tref{l:21}. 
In \sref{s:6}, we give a sufficient condition for \thetag{H.3}, 
which is the most important condition in \tref{l:21}. 
In \sref{s:7}, we prove \tref{l:22}, which is the special case $ d=1$ in \tref{l:21}, 
and we will give a convenient sufficient condition for \thetag{H.3} in this case.

\section{Set-up and main results}\label{s:2} 
Let $ S $ be a closed set in $  \mathbb{R}^{d}  $ such that $ 0 \in  S $ and $ \overline{ S ^{\text{int}}} =  S $, 
where $  S ^{\text{int}} $ means the interior of $  S $. 
Let 
$ \SSS = \{ \mathsf{s} = \sum _i \delta _{s_i}\, ;\, \mathsf{s}(K )< \infty 
\text{ for any compact set } K \} $, where $ \{ s_i \} $ is a sequence in $  S $.  
Then $ \SSS $ is the set of configurations on $  S $ by definition. 
We endow $ \SSS $ with the vague topology, 
under which $ \SSS $ is a Polish space. 

Let $ \mu $ be a probability measure on $ (\SSS , \mathcal{B}(\SSS ) ) $. 
We call a function $ \rho ^n $ the $ n $-correlation function of $ \mu $ 
with respect to (w.r.t.) the Lebesgue measure  
if $ \map{\rho ^n }{ S ^n}{ \mathbb{R} } $ is a permutation invariant function such that 
\begin{align}\label{:20a}&
\int_{A_1^{k_1}\! \times \! \cdots \! \times \! A_m^{k_m}} 
\rho ^n (x_1,\ldots,x_n) dx_1\cdots dx_n 
 = \int _{\SSS } \prod _{i = 1}^{m} 
\frac{\mathsf{s} (A_i) ! }
{(\mathsf{s} (A_i) - k_i )!} d\mu
 \end{align}
for any sequence of disjoint bounded measurable subsets 
$ A_1,\ldots,A_m \subset  S $ and a sequence of natural numbers 
$ k_1,\ldots,k_m $ satisfying $ k_1+\cdots + k_m = n $.  
Here we set $ (\mathsf{s} (A_i) - k_i )!=\infty $ if $ (\mathsf{s} (A_i) - k_i ) < 0 $.

We assume $ \mu $ satisfies the following. 

\noindent
\thetag{H.1} The measure $ \mu $ has 
a locally bounded, $ n $-correlation function 
$ \rho ^n $ for each $ n \in  \mathbb{N}   $.

We introduce a Hamiltonian on a bounded Borel set $ A $ as follows. 
For Borel measurable functions 
$ \map{\Phi }{ S }{ \mathbb{R} \cup \{\infty \}} $ 
and 
$ \map{\Psi }{ S \! \times \!  S }{ \mathbb{R} \cup \{\infty \}} $ 
with $ \Psi (x,y) = \Psi (y,x) $, let
\begin{align} \label{:20b} 
& \mathcal{H}_{A } ^{\Phi , \Psi } 
(\mathsf{x}) = 
\sum_{x_i\in A } \Phi ( x_i ) + 
\sum_{x_i, x_j\in A , i < j } 
\Psi ( x_i, x_j)
,\quad \text{ where } 
\mathsf{x} = \sum _i \delta _{x_i} 
.\end{align}
We assume $ \Phi < \infty $ almost everywhere (a.e.) to avoid triviality. 

For two measures $ \nu _1,\nu _2 $ on a measurable space 
$ (\Omega , \mathcal{B})$, we write $ \nu _1 \le \nu _2 $ 
if $ \nu _1(A)\le \nu _2(A)$, for all $ A\in\mathcal{B}$. 
We say a sequence of finite Radon measures $ \{ \nu ^{\n }  \} $ 
on a Polish space $ \Omega $ converge weakly to a finite Radon measure $ \nu $ 
if $ \lim_{\n \to \infty } \int f d \nu ^{\n }  = \int f d\nu $, 
for all $ f \in C_b(\Omega ) $. 

Throughout this paper, $ \{ \br \}$ denotes 
an increasing sequence of natural numbers. We set 
\begin{align}\label{:qg0}&
\Sr = \{ s \in  S \,;\,|s| < \br \},\quad 
\Srm = \{\mathsf{s}\in\SSS ;\mathsf{s}(\Sr ) = m \}
.\end{align}
For notational brevity, we suppress the dependence of $ \Sr $ on $ \{ \br \} $. 
We will later introduce 
$ \Srhat = \{ x \in  S \, ;\, |x|< r \}$ in \eqref{:21z}. 
By definition $ \Sr = \tilde{ S }_{b_r} $. 
In the proof of the main theorems, we will use $ \Sr $ 
more frequently than $ \Srhat $, which is the reason 
we have assigned the more complicated notation $ \Srhat $ 
to the simpler object $  \{ x \in  S \, ;\, |x|< r \}$. 
We set 
\begin{align}\label{:qg9}&
 \mathcal{H}_{r}(\mathsf{x}) = \mathcal{H}_{\Sr }^{\Phi ,\Psi }(\mathsf{x})
.\end{align}
For a subset $ A \subset  S $, we define the map 
$ \map{\pi _{A }}{\SSS }{\SSS } $ by 
$ \pi _{A } (\mathsf{s}) = \mathsf{s}( A \cap \cdot ) $.

\begin{dfn}\label{dfn:1} 
A probability measure $ \mu $ 
is said to be a $ (\Phi , \Psi ) $-quasi-Gibbs measure 
if the following holds: \\
\thetag{1} 
There exists an increasing sequence $ \{ \br \} $ of natural numbers 
such that, for each $ r ,m \in  \mathbb{N}   $, 
there exists a sequence of Borel subsets $ \mathsf{S}_{r,k}^{m} $ satisfying 
\begin{align} \label{:qg1}&
\mathsf{S}_{r,k}^{m} \subset  
\mathsf{S}_{r,k+1}^{m} \subset  \mathsf{S}_{r}^{m} 
\text{ for all } k , \quad 
%\\ \label{:qg3} &
\limi{k} \murk  =  \murm  
\quad \text{ weakly, } 
\end{align}
where $ \murk = \mu ( \cdot \cap \mathsf{S}_{r,k}^{m} ) $ and 
$ \murm = \mu (\cdot \cap \Srm ) $. \\
\thetag{2} 
For all $ r,m,k \in  \mathbb{N}   $ and 
$ \murk $-a.e.\! $ \mathsf{s} \in \SSS $, 
\begin{align}\label{:qg2}&
\frac{1}{\cref{;2y}}
e^{-\mathcal{H}_{r}(\mathsf{x})} 
1_{\Srm }(\mathsf{x})
\Lambda (d\mathsf{x})\le 
\murky (d\mathsf{x}) \le 
\cref{;2y} 
 e^{-\mathcal{H}_{r}(\mathsf{x})} 
1_{\Srm }(\mathsf{x})
\Lambda (d\mathsf{x})
.\end{align}
Here, $\Ct \label{;2y} = \cref{;2y}  (r,m,k,\pi _{\Sr ^c}(\mathsf{s})) $
 is a positive constant, 
$ \Lambda $ is the Poisson RPF for which the intensity 
is the Lebesgue measure on $  S $, and $ \murky $ is 
the regular conditional probability measure of $ \murk $ defined by 
\begin{align} \label{:qg3} 
& \murky (d\mathsf{x}) = \murk (\pi _{ \Sr } 
\in d\mathsf{x} | \ 
\pi _{ \Sr ^c }(\mathsf{s})) 
.\end{align}
\end{dfn}

We remark that the original definition of the quasi-Gibbs property 
in \cite{o.rm} is slightly more general than the above. 

We call $ \Phi $ (resp. $ \Psi $) a free (interaction) potential. 
When $ \Psi $ is an interaction potential, we implicitly assume that 
$ \Psi (x,y)= \Psi (y,x)$. 

\smallskip 

\begin{rem}\label{r:1} 
\thetag{1} 
By definition,  $ \murk ((\Srm )^c) = 0 $. 
Since $ \murky $ is 
$ \sigma [\pi _{ \Sr ^c }]$-measurable in 
$ \mathsf{s}$, 
we have the disintegration of the measure $\murk $ 
\begin{align}\label{:qg4}&
\murk \circ \pi _{ \Sr }^{-1}(d\mathsf{x}) = \int_{\SSS }
\murky (d\mathsf{x})\murk (d \mathsf{s})
.\end{align}
\thetag{2} 
Let 
$ \mury (d\mathsf{x}) = \murm (\pi _{ \Sr } (\mathsf{s}) 
\in d\mathsf{x} | \ \pi _{ \Sr ^c }(\mathsf{s})) $. 
Recall that a probability measure $\mu $ is said to be 
a $ (\Phi , \Psi ) $-canonical Gibbs measure if 
$ \mu $ satisfies the DLR equation \eqref{:qg5}; that is, 
for each $ r , m \in  \mathbb{N}   $, the conditional probability $ \mury $ satisfies 
\begin{align}\label{:qg5}&
\mury (d\mathsf{x}) = 
\frac{1}{\cref{;2yy}} e^{-\mathcal{H}_{r}(\mathsf{x}) - \Psi _r (\mathsf{x},\mathsf{s})  }
1_{\Srm }(\mathsf{x}) \Lambda (d\mathsf{x})
\quad \text{ for $ \murm $-a.e.\ $ \mathsf{s} $. }
\end{align}
Here, $ 0 < \Ct \label{;2yy}< \infty $ is the normalization and, 
for $ \mathsf{x} = \sum _i \delta _{x_i}$ and 
$ \mathsf{s} = \sum _j \delta _{s_j}$, we set 
\begin{align}\label{:qg6}&
\Psi _r (\mathsf{x},\mathsf{s})  = 
\sum_{x_i\in \Sr , s_j \in \Sr ^c } 
\Psi ( x_i, s_j) 
.\end{align}
\thetag{3} 
$ (\Phi , \Psi ) $-canonical Gibbs measures are 
$ (\Phi , \Psi ) $-quasi-Gibbs measures. The converse is, however,  not true. 
When $ \Psi (x,y) = - \beta \log|x-y|$ and the $ \mu $ are translation invariant, 
the $ \mu $ are not $ (\Phi , \Psi ) $-canonical Gibbs measures. 
This is because the DLR equation does not make sense. Indeed, 
$ |
\Psi _r (\mathsf{x},\mathsf{s})  
| = \infty $ for $ \mu $-almost surely (a.s.) $ \mathsf{s}$. 
The point is that one can expect a cancellation between 
$ \cref{;2yy} $ and $ e^{-\Psi _r (\mathsf{x},\mathsf{s})  } $ even if $ |\Psi _r (\mathsf{x},\mathsf{s})  | = \infty $. 
\\
\thetag{4} Unlike canonical Gibbs measures, 
the notion of quasi-Gibbs measures is quite flexible for free potentials. 
Indeed, if $ \mu $ is a $ (\Phi ,\Psi )$-quasi-Gibbs measure, 
then $ \mu $ is also a $ (\Phi + F ,\Psi )$-quasi-Gibbs measure 
for any locally bounded measurable function $ F $. 
Thus, we write $ \mu $ a $ \Psi $-quasi-Gibbs measure if 
$ \mu $ is a $ (0,\Psi )$-quasi-Gibbs measure. 
\end{rem}

We give a pair of conditions for the quasi-Gibbs property. 
These conditions guarantee that $ \mu $ has 
a good finite-particle approximation 
$\{\muN \}_{\n \in \mathbb{N}    }$ that enables us to prove the quasi-Gibbs property. 
We set  
\begin{align}\label{:21z}&
 \Srhat = \{ x \in  S \, ;\, |x|< r \} ,\quad \Dr ^n = \prod_{m=1}^{n} \{ |x_m|< r \} 
.\end{align}

\medskip 
\noindent 
\thetag{H.2} 
There exists a sequence of probability measures $\{ \muN \}_{\n \in \mathbb{N}    }$ on 
$ \SSS $ satisfying the following.  

\noindent 
\thetag{1} 
The $ n $-correlation functions $  \rN $ of $  \muN $ satisfy  
\begin{align} \label{:21a} &
\lim _{\n \to \infty } \rN (\mathbf{x}_n) = \rho ^n (\mathbf{x}_n) 
\quad \text{ a.e.} 
\quad \text{ for all $ n \in  \mathbb{N}   $,}
\\ \label{:21b} 
& \sup 
\{  \rN (\mathbf{x}_n)  ; \n \in \mathbb{N}    ,\, \mathbf{x}_n \in  \Srhatn \} 
\le \{ \cref{;70} n ^{\delta }\} ^n
\quad \text{ for all $ n,r\in  \mathbb{N}   $} 
,\end{align}
where $ \mathbf{x}_n = (x_1,\ldots,x_n) \in S ^{n}$, 
$ \Ct \label{;70}=\cref{;70}(r) >0$, and $ \delta = \delta (r) < 1 $ 
are constants depending on $ r \in  \mathbb{N}   $. 

\noindent 
\thetag{2} \  $ \muN (\mathsf{s}( S ) \le \nN ) = 1 $ for each $ \n $, 
where $ \nN \in  \mathbb{N}   $. 

\noindent \thetag{3} 
$ \muN $ is a $ (\PhiN ,\PsiN ) $-canonical Gibbs measure. 

\noindent \thetag{4} 
There exists a sequence $ \{ \mmi  \}_{\n \in  \mathbb{N}   } $ in $  \mathbb{R}^{d}  $ such that 
\begin{align}\label{:21c}&
\limi{\n } \{\PhiN (x) - \mmi \cdot x \} = 
\Phi (x) \quad \text{ for a.e.\ \!\! $ x $,} 
\quad  
\\ \notag &
\infN \inf _{x \in  S } \{\PhiN (x) - \mmi \cdot x \}  > -\infty 
.\end{align}
Here $ \cdot $ denotes the standard inner product in $  \mathbb{R}^{d}  $. 

\noindent \thetag{5} 
The interaction potentials 
	$ \map{\PsiN }{ S \! \times \!  S }{ \mathbb{R} \cup \{ \infty \}} $ 
satisfy the following. 
\begin{align} 
 \label{:21d}& 
\limi{\n } \PsiN = \Psi \text{ compactly and uniformly in }
 C^1( S \! \times \! S \backslash \{ x=y \}) , 
\\ \notag & 
%\PsiN \in C^1( S \! \times \! S \backslash \{ x=y \}) ,\quad 
	\infN \inf _{  x,y \in \Srhat }\PsiN (x,y) > - \infty 
		\quad \text{ for all }r\in \mathbb{N}   
.\end{align}

\begin{rem}\label{r:22} 
For the GUE soft-edge approximation of the Airy RPF, 
we take $ \mmi  = \n ^{1/3} $. In fact, in this case, the limit of $ \PhiN $ diverges. 
Hence, we substitute $ \mmi \cdot x $ from $ \PhiN (x) $ to make the limit finite. 
In a forthcoming paper, we will see 
that the terms $ \mmi \cdot x $ are cancelled by the interaction terms. 
\end{rem}

\bigskip 

The next assumption \thetag{H.3} is a tightness condition on $\{\muN \}$ 
according to the interaction $ \PsiN $. Indeed, \thetag{H.3} plays 
the most significant role in the proof of the quasi-Gibbs property of $ \mu $. 
To introduce \thetag{H.3}, we establish some notations. 

Let 
$ \mathsf{x} = \sum \delta _{x_i}$ and $ \mathsf{y} = \sum \delta _{y_j}\in \SSS $. 
For $ \{\Sr \}$ in \eqref{:qg0}, we set 
$ \Srs = \Ss \backslash \Sr $ and $ \Srinfty =  \Sr ^c $. 
For $r<s \le t < u \le \infty $, we set 
\begin{align} \label{:21e}&\quad \quad 
\PsiNrstu (\mathsf{x},\mathsf{y})  = 
\sum_{x_i\in  S _{rs} ,\ y_j \in  S _{tu}} 
\PsiN (x_i,y_j) \quad \quad 
.\end{align}
We write $\PsiN _{r,st} = \PsiN _{0r,st}$  and 
$ \PsiN _{r,rs} (\mathsf{x},\mathsf{y}) = \PsiN _{r,rs} (x,\mathsf{y}) $ if 
$ \mathsf{x}=\delta _{x}$. 

For $r<s \le t < u \le \infty $, let 
\begin{align}\label{:21o}&
\wPsiNrstu (\mathsf{x},\mathsf{y}) = 
\PsiNrstu (\mathsf{x},\mathsf{y})  + \{ \xrs \} \cdot (\mmt -\mmu ) 
.\end{align}
We set $\wPsiN _{r,st} = \wPsiN _{0r,st}$. 
For $ \{ \PsiN \} $, $ r,k\in  \mathbb{N}   $, and $\{ \mms \} $ we define $ \Hrk $ by 
\begin{align}
\label{:21f} &
\Hrk = \{ \mathsf{y}\in \mathsf{S} \, ;\, \{ 
\supN \sups \sup _{x\not=w\in\Sr }   
\frac{| \wPsiN _{r,rs} (x,\mathsf{y}) -\wPsiN _{r,rs} (w,\mathsf{y}) |} 
{| x -  w |} \} 
  \le \ak \} 
.\end{align}

We note that the set $ \Hrk $ depends on $ \{ \mms  \} $, 
although for brevity, we suppress $ \{ \mms  \} $ in denoting $ \Hrk $. 
The functions $ \{ \mms  \} $ in \eqref{:21f} are related to 
the sequence $ \{ \mmi  \} $ in \thetag{4} of \thetag{H.2} 
by the condition \eqref{:21h} below.  

\medskip 

\noindent
\thetag{H.3} 
There exists a sequence $ \{ \mms \} $ in $  \mathbb{R}^{d}  $ 
such that the set $ \Hrk $ satisfies the following: 
\begin{align} \label{:21g} & 
\limi{k} \limsup_{ \n \to \infty } 
\muN (\Hrk ^c) = 0 \quad \text{ for all $ r \in  \mathbb{N}   $}
,\\\label{:21h}&
\limi{s} \mms =  \mmi ,\quad 
\supN |\mms | < \infty \quad \text{ for all } s\in \mathbb{N}
. \end{align}

\medskip

\begin{rem}\label{r:21}
When $ \mms \equiv 0 $, the set $ \Hrk $ in \eqref{:21f} equals 
$ \Hrk $ in \cite{o.rm}. Thus, this definition is a generalization of 
$ \Hrk $ in \cite{o.rm}. The function $ \mmi $ compensates the sum 
$ ({\PsiN _{r,rs} (x,\mathsf{y}) -\PsiN _{r,rs} (w,\mathsf{y})} ) / ({ x -  w }) $. 
For the Airy RPFs, we have no hope to ensure \thetag{H.3} 
without this compensation. 
\end{rem}

\begin{thm}	\label{l:21}
Assume \thetag{H.1}, \thetag{H.2} and \thetag{H.3}. 
Then $ \mu $ is a $ (\Phi , \Psi ) $-quasi-Gibbs measure. 
\end{thm}

We next assume $ d = 1,2$. Thus, to unify these two cases, 
we set $  S = \mathbb{C}$. Indeed, we regard here $ \mathbb{R}^2 $ 
as $ \mathbb{C}$ by the natural correspondence: 
$ \mathbb{R}^2\ni (x,y) \mapsto x + \sqrt{-1}y \in \mathbb{C}$, 
and $ \mathbb{R}$ as the real axis in $ \mathbb{C}$. 
Hence, we view 
$ \mathsf{m}_{r}^{\n } = 
(\mathsf{m}_{r,1}^{\n } , \mathsf{m}_{r,2}^{\n } ) \in \mathbb{R}^2 $ as 
$ \mathsf{m}_{r}^{\n } = \mathsf{m}_{r,1}^{\n } + \sqrt{-1} \mathsf{m}_{r,2}^{\n } \in \mathbb{C}$. 

We assume $ \PsiN $ is independent of $ \n $ and of the form 
\begin{align}\label{:22a}
\Psi (x,y)  & := \PsiN (x,y) = - \beta \log |x -y |\quad \quad (\beta \in  \mathbb{R} )
.\end{align}
We will give a sufficient condition of \thetag{H.3} in terms of correlation functions. 

Let $ \xx =\sum _{i}\delta_{\x _i }$ and $ \Drs = \Ds \backslash \Dr $, where $ \Dr = \{ s \in  S ;\ |s|< r \}$, as before. 
For $1 \le r < s \le \infty $ let $ \map{\mathsf{v}_{\ell ,rs} }{\SSS }{\mathbb{C}}$ such that 
\begin{align}\label{:22b}&
\mathsf{v}_{\ell ,rs} (\xx )= \beta 
\big\{\sum _{\x _i \in \Drs }
 \frac{1}{{ x }_i^{\ell } } \big\} \quad (\ell \ge 2) 
 \\ \label{:22c} & 
 \mathsf{v}_{1,rs} (\xx ) = 
 \beta \big\{ \sum_{ x _i \in \Drs } 
\frac{1}{ x _i } \big\} 
 + \bar{\mathsf{m}}_{r}^{\n } -  \bar{\mathsf{m}}_{s}^{\n } 
 \quad (\ell = 1) 
.\end{align}
Here 
$ \bar{\mathsf{m}}_{r}^{\n } = \mathsf{m}_{r,1}^{\n } - \sqrt{-1} \mathsf{m}_{r,2}^{\n }$ 
is the complex conjugate of $ \mathsf{m}_{r}^{\n } $. 

When $ \ell = 1$, $  \mathsf{v}_{1,rs} $ depends on $ \n $. 
Hence, we write $ \mathsf{v}_{1,rs} = \mathsf{v}_{1,rs}^{\n }$ when 
we emphasize the dependence on $ \n $. 
Although $ \mathsf{v}_{\ell ,rs} $ are independent of $ \n $ when $ \ell \ge 2 $, 
we write $ \mathsf{v}_{\ell ,rs} ^{\n }$ if we want to unify the notation (see \eqref{:22e} and \eqref{:22f}). 

Note that the sum in \eqref{:22b} makes sense for $ \muN $-a.s. $ \xx $ 
even if $ s=\infty $. 
Indeed, by \thetag{2} of \thetag{H.2}, the total number of particles has 
deterministic bound $ \nN $ under $ \muN $. 
Hence, $ \mathsf{v}_{\ell ,rs} (\xx )$ is well defined and finite for 
$ \muN $-a.s. $ \xx $, for all $ \n \in  \mathbb{N}   $. 

Now the key assumption is as follows. 

\smallskip 
\noindent 
\thetag{H.4} There exists an $ \ellell $ such that $ 2\le \ellell \in \mathbb{N}$ 
and that 
\begin{align} \label{:22d}&
\supN  \{ \int_{1\le |x |<\infty } 
 \frac{1 }{\ |  x |^{\ellell }} \, \rNone (x )dx \} < \infty 
\\\intertext{and that, for each $  1 \le \ell < \ellell $, }
\label{:22e}&
\supN  \|  \supP | \vellrs ^{\mathsf{p}} | \,  \|_{\LmNone } < \infty \text{ for all } r<s \in \mathbb{N}, 
\\\label{:22f} &
\limi{s} \supN  \|  \supP | \vellsi ^{\mathsf{p}} | \,  \|_{\LmNone } = 0 
.\end{align}

\begin{thm}	\label{l:22} 
Assume \eqref{:22a} and $  S = \mathbb{C}$. 
Assume \thetag{H.1}, \thetag{H.2} and \thetag{H.4}. 
Then $ \mu $ is a $ (\Phi , \Psi ) $-quasi-Gibbs measure. 
\end{thm}

In a forthcoming paper, 
we will prove the quasi-Gibbs property of Airy$ _{\beta }$ RPFs, 
and solve the associated ISDEs. \tref{l:22} will be used there. 
Whenever we consider the RPFs appearing under soft-edge scaling, 
such as Tacknode \cite{johansson.11}, 
the divergence of the free potentials such as 
\eqref{:14} always occurs, which causes a difficulty in treating soft-edge scaling. 
It is plausible that our results can resolve this. 

Stochastic dynamics of infinitely many particle systems in $ \mathbb{R}$ 
related to 
random matrix theory have been constructed by explicit calculation 
based on space-time correlation functions 
(see \cite{johansson.03}, \cite{knt.04}, \cite{kt.07-ptrf}, \cite{kt.07}, 
\cite{p-spohn.02}, and others). 
In this body of work, the properties of dynamics from a viewpoint of stochastic analysis, 
such as the semi-martingale property, and Ito's formula, have not yet been 
well developed. Our method, together with the forthcoming paper, 
gives SDE representations of the dynamics, 
which enables us to use the stochastic analysis effectively. 

In \cite{yuu.05}, Yuu proved that all determinantal RPFs with kernels $ K $ 
such that the spectrum $ \mathrm{Spec}(K)$ of the associated $ L^2$-operator 
satisfies $ 0 <  \mathrm{Spec}(K) < 1$ become a kind of Gibbs measure, 
and by using this, he constructed associated diffusions. 
However, the spectrum of kernels of determinantal RPFs 
appearing in random matrix theory in the infinite-volume limit 
usually contains $ 1 $.  
Hence, his result excludes RPFs related to random matrix theory 
such as Dyson's model, the Bessel RPF, and, in particular, the Airy RPF. 
It is an interesting open problem to prove that all determinantal RPFs are quasi-Gibbs measures. 

%%%%%%%%%%%[][]ここから8/14
\section{Proof of \tref{l:21}} \label{s:3}
In \sref{s:3}--\sref{s:5}, we will prove \tref{l:21}. 
In the present section, we first prepare a lemma from \cite{o.rm}, 
and explain the strategy of the proof of \tref{l:21}.  
In fact, we divide the proof into two parts. 
We will prove the first step \eqref{:35a} in \sref{s:4}, 
and the second step \eqref{:35b} in \sref{s:5}. 

We fix $ r , m \in  \mathbb{N}   $ throughout Sections \ref{s:3}--\ref{s:5}. 
Let $ \Srm  $ be as in \eqref{:qg0}. 
Using the set $ \Hrk $ defined in \eqref{:21f}, 
we introduce cut-off measures $ \mukNm $: 
\begin{align}\label{:30a}& 
\mukNm = \muN (\cdot \cap \Srm \cap \Hrk ) 
.\end{align}
We will prove \tref{l:21} along this sequence $ \{ \mukNm \} $. 
For this, we first note the following. 

\begin{lem}[Lemma 4.2 in \cite{o.rm}] \label{l:31} 
There exists a weak convergent subsequence of $ \{ \mukNm \}$, 
denoted by the same symbol, with limit measures $ \{ \murk \}$ 
satisfying \eqref{:qg1}, for all $ r,k,m $. 
\end{lem}

Let $ \muABN  $ denote the conditional probability of 
$ \mukNm $ defined by 
\begin{align}\label{:32a}&
\muABN (d\mathsf{x}) = 
\mukNm (\pi _{ \Sr } \in d \mathsf{x} | \, 
\pi _{ \Srs } (\mathsf{s}))
.\end{align}
We note that, although $ \mukNm $ is not necessarily a probability measure, 
we normalize it in such a way that the conditional measure $ \muABN  $ 
is a probability measure. As a result, we have $ \muABN (\SSS ) = 1$ and 
\begin{align}\label{:32b}&
\mukNm \circ \pi _{ \Sr }^{-1}(d\mathsf{x}) = 
\int _{\SSS }\muABN (d\mathsf{x})\, \mukNm \circ \pi _{ \Srs }^{-1}(d\mathsf{s})
.\end{align}
Recall that by \thetag{H.2}, $ \muN $ is a $ (\PhiN ,\PsiN ) $-canonical Gibbs measure. 
Then $ \muN $ satisfies the DLR equation \eqref{:qg5}. 
Hence, $ \muABN $ is absolutely continuous w.r.t.\ 
$ e^{-\mathcal{H}_{r}^{\n }(\mathsf{x})}\mA $. 
Therefore, we denote its density by $ \sABN $. 
Then by definition, we have for $\mukNm $-a.e.\ $\mathsf{s}$ 
\begin{align} \label{:32c} &
\quad \quad \sABN (\mathsf{x}) e^{-\mathcal{H}_{r}^{\n }(\mathsf{x})}\mA 
 = \muABN (d\mathsf{x}),\quad \quad \text{ where }
\mathcal{H}_{r}^{\n } = \mathcal{H}_{\Sr }^{\Phi ^{\n },\PsiN }
.\end{align}

We recall that the limit $ \limi{\n } \PhiN $ diverges in general. 
Such a divergence implies that for $ \mathcal{H}_{r}^{\n } $. 
Thus, to prevent this, we consider the compensation constant $ \mmi $ in \eqref{:21c}, and 
set 
\begin{align}\label{:32d}&
 \widetilde{\mathcal{H}}_{r}^{\n } (\mathsf{x}) = 
 \sum_{x_i\in \Sr } \{ \PhiN ( x_i ) - \mmi \cdot x_i \} + 
 \sum_{x_i, x_j\in \Sr , i < j } \PsiN ( x_i, x_j)  
, \end{align}
where $ \mathsf{x}=\sum_i \delta_{x_i}$. 

\begin{lem} \label{l:32} $  \widetilde{\mathcal{H}}_{r}^{\n } $ satisfy the following. 
\begin{align} \label{:32q}&
\limi{\n } e^{-\widetilde{\mathcal{H}}_{r}^{\n } (\mathsf{x})} = 
e^{-\mathcal{H}_{r}(\mathsf{x})} \quad \text{ for $ \mu $-a.e.\ $ \mathsf{x}$}
,\\ \label{:32r}&
\{\sup_{\n \in \mathbb{N}   } \sup_{\mathsf{x}\in \Srm }
e^{-\widetilde{\mathcal{H}}_{r}^{\n } (\mathsf{x})} \} < \infty 
\quad \text{ for each } m \in  \mathbb{N}   
.\end{align}
\end{lem}
\begin{proof}
\lref{l:32} follows immediately from \eqref{:21c} and \eqref{:21d} 
combined with \eqref{:32d}. 
\end{proof}

We estimate the Boltzmann constants for the Hamiltonians $ \widetilde{\mathcal{H}}_{r}^{\n } $. 
\begin{lem} \label{l:33} 
Let $ \Ct (n)\label{;33}$ be the constant defined by 
\begin{align}\notag & 
\cref{;33} (n) = \sup_{n\le \n \in \mathbb{N}    } \max \{ 
 \int _{\Srm } e^{-\widetilde{\mathcal{H}}_{r}^{\n } (\mathsf{x})  }
 \mA , \ 
\frac{1}
{ \int _{\Srm } e^{-\widetilde{\mathcal{H}}_{r}^{\n } (\mathsf{x}) }  \mA } \} 
.\end{align}
Here $ \mathsf{x} =  \sum_{i=1}^{m} \delta_{x_i}  $
Then there exists an $ \n _0 $ such that $ \cref{;33}(\n _0 ) < \infty $. 
\end{lem}
\begin{proof} 
We deduce  
from \eqref{:32q}, \eqref{:32r}, and the bounded convergence theorem that 
\begin{align} & \notag 
\lim_{\n \to\infty} \int _{\Srm }
e^{-\widetilde{\mathcal{H}}_{r}^{\n } (\mathsf{x}) } \mA = 
\int _{\Srm } e^{-\mathcal{H}_{r}(\mathsf{x})}\mA 
< \infty 
.\end{align}

Recall that $ \Phi (x) < \infty $ a.e.\ by assumption 
(see the line after \eqref{:20b}) 
and $ \Psi (x,y)< \infty $ a.e.\ by the first assumption of \eqref{:21d}.  
Then we see that $ \mathcal{H}_{r}(\mathsf{x})<\infty $ a.e.. 
Hence, 
\begin{align*}&
\int _{\Srm }e^{-\mathcal{H}_{r}(\mathsf{x})}\mA > 0 
.\end{align*}
Combining these completes the proof. 
\end{proof}

Taking \lref{l:32} and \lref{l:33} into account 
we consider the Radon-Nikodym density $  \st \ABN $ 
of $ \muABN $ w.r.t.\ $ e^{-\widetilde{\mathcal{H}}_{r}^{\n } (\mathsf{x})}\mA $. 
Namely, by definition we have 
\begin{align}\label{34y}&
 \st \ABN (\mathsf{x}) e^{-\widetilde{\mathcal{H}}_{r}^{\n } (\mathsf{x})} 
 \mA 
 = \muABN (d\mathsf{x})
. \end{align}
It is then clearly seen that with normalization $ \Ct \label{;42a}$ 
\begin{align}\label{:34z}&
 \st \ABN (\mathsf{x}) = 
 \frac{1}{\cref{;42a}} e^{ - \mmi \cdot \xr } \sABN (\mathsf{x}) 
\quad \text{ for } \mathsf{x} = \sum_i \delta_{x_i} 
.\end{align}

We next consider the decomposition of $ \st \ABN $ in \eqref{:34z}.  
\begin{lem} \label{l:34} 
The density $\st \ABN $ is expressed in such a way that 
\begin{align} \label{:34a} 
 \st \ABN (\mathsf{x}) & = 
 \frac
 {1} {{\cref{;42}^{\n }(\mathsf{s})}} 
 e^{ - \mmr \cdot \xr - \wPsiNrs (\mathsf{x},\mathsf{s})} {\tCNtilde }
 \quad \text{ for $\mukNm $-a.e.\ $\mathsf{s} $}
.\end{align}
Here $ \wPsiNrs $ were given by \eqref{:21z}, and 
$ \Ct ^{\n }\label{;42}(\mathsf{s})$ is the normalization 
\begin{align} \label{:34b}&
\cref{;42}^{\n }(\mathsf{s}) = \int _{\SSS }
e^{ - \mmr\cdot \xr  -\wPsiNrs (\mathsf{x},\mathsf{s})} \tCNtilde 
e^{-\widetilde{\mathcal{H}}_{r}^{\n } (\mathsf{x})  } \mA 
,\end{align}
and $\tCNtilde $ is defined by  
\begin{align}\label{:34c}
\tCNtilde = 1_{\Srm }(\mathsf{x}) \int _{\SSS }&
1_{\Hrk }(\pi _{\Srs }(\mathsf{s})+
\mathsf{z})
\\ \notag & \cdot 
e^{ -\wPsiNrsi (\mathsf{x},\mathsf{z})- \wPsiNssi (\mathsf{s},\mathsf{z})}
\mukNm \circ \pi _{\Ssinfty }^{-1} (d\mathsf{z})  
.\end{align}
\end{lem}

\begin{proof}
 \lref{l:34} is immediate from \eqref{:32a} and \eqref{:32c}.  
 Indeed, recall that $ \muN $ is a $ (\PhiN ,\PsiN ) $-canonical Gibbs measure 
 by the assumption \thetag{2} of \thetag{H.2}. 
Then from this and noting \eqref{:30a}, we deduce that 
the Radon-Nikodym density $ \sABN $ given by \eqref{:32c} satisfies 
\begin{align}\label{:34p}&
\sABN (\mathsf{x}) = \mathrm{const.}\  e^{- \PsiNrs (\mathsf{x},\mathsf{s}) } {\tCN } 
,\end{align}
where $\tCN $ is defined by  
\begin{align} \label{:34f}
\tCN = 1_{\Srm }(\mathsf{x}) \int _{\SSS }&
1_{\Hrk }(\pi _{\Srs }(\mathsf{s})+\mathsf{z}) 
\\&  \notag 
\cdot 
e^{ -\PsiNrsi (\mathsf{x},\mathsf{z})- \PsiNssi (\mathsf{s},\mathsf{z})}
\mukNm \circ \pi _{\Ssinfty }^{-1} (d\mathsf{z})  
.\end{align} 
We deduce from \eqref{:34z} and \eqref{:34p} that, 
 for $\mukNm $-a.e.\ $\mathsf{s} $, 
\begin{align} \label{:34d} 
 \st \ABN (\mathsf{x}) & = 
 \frac{1}{\cref{;42a}}
 e^{ - \mmi \cdot \xr } \sABN (\mathsf{x}) 
  \\ \notag & =  \frac{1}{{\cref{;34}^{\n }(\mathsf{s})}} 
  e^{ - \mmi \cdot \xr }   
 e^{- \PsiNrs (\mathsf{x},\mathsf{s}) } {\tCN }
. \end{align}
Here $ \Ct ^{\n }\label{;34}(\mathsf{s})$ is the normalization 
\begin{align} \label{:34e}&
\cref{;34}^{\n }(\mathsf{s}) = \int _{\SSS }
e^{ - \mmi \cdot \xr  -\PsiNrs (\mathsf{x},\mathsf{s})} \tCN 
e^{-\widetilde{\mathcal{H}}_{r}^{\n } (\mathsf{x}) 
 } \mA 
.\end{align}
Therefore, we deduce \eqref{:34a} 
from \eqref{:34d} combined with \eqref{:34c} and 
\eqref{:34f}, and the definition of $  \wPsiNrs $.  
\end{proof}

The quasi-Gibbs property consists of two conditions: 
\eqref{:qg1} and \eqref{:qg2}. We have already proved \eqref{:qg1} by \lref{l:31}. 
Therefore, it only remains to prove \eqref{:qg2}. This task will be carried out 
in the next two sections. We now explain the strategy of the proof of \eqref{:qg2}.  

By taking the representation \eqref{:32b} into account, 
the proof consists of two kinds of limit procedure: 
\eqref{:35a} $ \n \to \infty $ and then \eqref{:35b}  $ s \to \infty $, 
which involve the following convergence. 
\begin{align}& \label{:35a} 
\limi{\n } \muABN =  \muABs , \quad \quad 
\limi{\n } \mukNm \circ \pi _{ \Srs }^{-1} = \murk \circ \pi _{ \Srs }^{-1}
,\\ & \label{:35b}
\limi{s} \muABs = \muA 
.\end{align}
Note that two of these are the convergence of the {\it conditional} measures. 
In comparing with the weak convergence of $ \{ \mukNm \}$ in \lref{l:31}, 
it is noted that the convergence of the conditional measures is much more delicate. 
It involves a variety of strong convergence of 
the conditioned variable $ \mathsf{s}$. 

In each step, we prove the bounds of the densities being uniform in $ \n , s $ 
(\eqref{:42a} and \eqref{:51a}) 
and the related quantities as well as the convergence of measures as above. 
The uniformity of the bounds is the crucial point of the proof. 
We emphasize that we can carry out the proof because we treat the cut-off measures 
$ \{ \mukNm  \} $ defined by \eqref{:30a}. This cut-off is done by the set $ \Hrk $. 
Therefore, the assumption \thetag{H.3} plays a significant role in the proof of \tref{l:21}. 

The first step consists of four lemmas. 
Recall the expressions \eqref{:32b} and \eqref{34y}. 
We have already proved the uniform bound of 
$  \int _{\Srm } e^{-\widetilde{\mathcal{H}}_{r}^{\n } (\mathsf{x})}\mA $ 
in \lref{l:33}, and will prove that for $ \st \ABN $ in \lref{l:42}. 
We then prove weak convergence 
$ \limi{\n } \mukNm \circ \pi _{ \Srs }^{-1} = \murk \circ \pi _{ \Srs }^{-1}$
 and $ L^1$ convergence of their densities (\lref{l:43}, \pref{l:45}). 
In this schema, we will have to prove the convergence of both 
 $  \st \ABN (\mathsf{x})  $ and $ e^{-\widetilde{\mathcal{H}}_{r}^{\n } (\mathsf{x})} $. 
Since the convergence of $ e^{-\widetilde{\mathcal{H}}_{r}^{\n } (\mathsf{x})} $ 
has been done by \lref{l:32} and \lref{l:33}, 
we will concentrate on that for $  \st \ABN (\mathsf{x})  $. 

The second step consists of two lemmas. In \lref{l:51}, we prove the absolute continuity of the measures $ \muABs $ and the uniform bound \eqref{:51a} of their densities $ \skABs (\mathsf{x})$. 
Finally, in \lref{l:52}, we prove the convergence of 
$ \skABs (\mathsf{x}) $ as $ s\to \infty $ 
 using martingale convergence theorems to complete the proof of 
the quasi-Gibbs property.

\section{Proof of the first step.\label{s:4}}
In \lref{l:42}, we will give both sides bounds of $  \st \ABN (\mathsf{x}) $. 
For this purpose, we control the sum of the interactions in 
\eqref{:21e} and \eqref{:21o} . 
We begin by setting  
\begin{align}\label{:41z}&
d_{\mathsf{S}^{n}_{rs}}(\mathsf{s},\mathsf{t}) = \min \{ \sum_{i=1}^n |s_i-t_i| \} 
\quad \text{ for $ \mathsf{s}, \mathsf{t}\in \mathsf{S}^{n}_{rs} $}
,\end{align}
where the minimum is taken over the labeling such that 
$ \pi_{\Srs }(\mathsf{s}) = \sum_{i=1}^n \delta _{s_i} $ and  
$ \pi_{\Srs }(\mathsf{t}) = \sum_{i=1}^n \delta _{t_i} $. 

\begin{lem} \label{l:41} 
\thetag{1} 
Set $\Ct \label{;44} (k) =  mk \cdot \mathrm{diam}\, (\Sr ) $.  
Then,  for each $ k\in \mathbb{N}   $,  
\begin{align}\label{:41a}&
\supN \sup _{r\le s < t \in \mathbb{N}   } 
\sup_{\mathsf{x}, \mathsf{x}' \in\Srm }
\sup_{\mathsf{s}\in \Hrk }
| \wPsiNrst (\mathsf{x},\mathsf{s}) -\wPsiNrst (\mathsf{x}',\mathsf{s}) |
\le \cref{;44}
.\end{align}  
\thetag{2} Let 
$ \mathsf{S}^{n}_{rs} = \{\mathsf{x}\in \SSS ;\,\mathsf{x}(\Srs ) = n \}$ 
and let $ \Hsl $ be as in \eqref{:21f}. Namely, we set 
\begin{align}\label{:41y}& 
\Hsl 
= \{ \mathsf{y}\in \mathsf{S} \, ;\, \{ 
\supN \sup _{s<t\in\mathbb{N}} \sup _{x\not=w\in S _{s}}  
\frac{| \wPsiN _{s,st} (x,\mathsf{y}) -\wPsiN _{s,st} (w,\mathsf{y}) |} 
{| x -  w |} \} 
  \le l \} 
.\end{align}
Then, for each $ n ,l  \in \mathbb{N}   $, 
\begin{align}\label{:41b} &
\supN \sup _{r\le s < t \in \mathbb{N}   } 
 \sup_{\mathsf{y},\mathsf{y}'\in \mathsf{S}^{n}_{rs} }
 \sup_{\mathsf{s}\in \Hsl }
\{ \frac
{| \wPsiNrsst (\mathsf{y},\mathsf{s}) - 
\wPsiNrsst (\mathsf{y}',\mathsf{s}) |}
{d_{\mathsf{S}^{n}_{rs}} (\mathsf{y},\mathsf{y}')} \} \le l 
.\end{align}
\\\thetag{3} 
For $ q \in  \mathbb{N}   $, we set $ B_r^{q}=\{0< |s-\Sr |<1/q \} $ and 
\begin{align}\label{:41c}&
 \Aq  =  \{ \mathsf{s} \in  \mathsf{S}^{n}_{rs}\cap \Hsl   \, ;  
 \, \mathsf{s}(B_r^{q})=0 \} 
.\end{align}
Let $ \Ct \label{;56b}=\cref{;56b} (mnq,rsl)$ be the constant defined by 
\begin{align}\label{:41d}&
\cref{;56b}   = \supN  \sup _{ {\mathsf{x}\in\Srm }} 
\sup \{ 
\frac 
{|\wPsiNrs (\mathsf{x},\mathsf{y})-\wPsiNrs (\mathsf{x},\mathsf{y}')| }%
{ d_{\mathsf{S}^{n}_{rs}}(\mathsf{y},\mathsf{y}')}  ; \, 
\mathsf{y} \not = \mathsf{y}'  \in   \Aq   \}
.\end{align}
Then we have $ \cref{;56b}   < \infty $. 
\end{lem}
\begin{proof}
\eqref{:41a} follows from \eqref{:21e}, \eqref{:21o}, and \eqref{:21f} immediately. 

We next prove \eqref{:41b}. 
Let $ \{ y_i \}_i^n $ and $ \{ y'_i \}_i^n $ be labels such that 
$  \pi_{\Srs }(\mathsf{y}) = \sum_{i=1}^n \delta _{y_i} $ and that 
$  \pi_{\Srs }(\mathsf{y}') = \sum_{i=1}^n \delta _{y'_i} $. 
Then we have 
\begin{align}\label{:41e}
\frac
{| \wPsiNrsst (\mathsf{y},\mathsf{s}) - \wPsiNrsst (\mathsf{y}',\mathsf{s}) |}
{\sum_{i=1}^n |y_i-y'_i|} 
&=
\frac
{| \sum_{i=1}^n \{\wPsiNrsst (y_i,\mathsf{s}) - \wPsiNrsst (y'_i,\mathsf{s}) \}|}
{\sum_{i=1}^n |y_i-y'_i|} 
\\ \notag & \le  
\frac
{\sum_{i=1}^n | \wPsiNrsst (y_i,\mathsf{s}) - \wPsiNrsst (y'_i,\mathsf{s}) | }
{\sum_{i=1}^n |y_i-y'_i|} 
\\ \notag & \le 
\max_{i=1,\ldots,n}\{\frac
{| \wPsiNrsst (y_i,\mathsf{s}) - \wPsiNrsst (y'_i,\mathsf{s}) | } { |y_i-y'_i|}   \} 
\\ \notag &
\le l 
.\end{align}
Here we used the inequality 
$ \{\sum_i^n a_i\} /\{ \sum_i^n b_i \} \le 
\max  \{ a_m/b_m ;m=1,\ldots,n \}$ valid for $ a_i \ge 0 $ and $ b_j > 0 $ 
in the third line. We also used \eqref{:21f} and $ \Srs \subset \Ss $ in the last line. 
Taking the maximum of the labels on the left-hand side of \eqref{:41e}, 
we obtain \eqref{:41b}.  

The proof of \eqref{:41d} is similar to \eqref{:41b}. Indeed, in the same fashion as above, 
we deduce that 
\begin{align}\label{:41f}
\frac
{| \wPsiNrst (\mathsf{x},\mathsf{y}) - \wPsiNrst (\mathsf{x},\mathsf{y}') |}
{\sum_{i=1}^n |y_i-y'_i|} 
& \le 
\max_{i=1,\ldots,n}\{\frac
{| \wPsiNrst (\mathsf{x},y_i) - \wPsiNrst (\mathsf{x},y'_i) | } { |y_i-y'_i|}   \} 
\\ \notag &
= \max_{i=1,\ldots,n}\{\frac
{| \PsiNrst (\mathsf{x},y_i) - \PsiNrst (\mathsf{x},y'_i)  | } { |y_i-y'_i|}   \} 
.\end{align}
Here the second line follows from 
$$ \wPsiNrst (\mathsf{x},y_i) - \wPsiNrst (\mathsf{x},y'_i) 
= \PsiNrst (\mathsf{x},y_i) - \PsiNrst (\mathsf{x},y'_i) 
.$$
Since $ \PsiN $ converge to $ \Psi $ compactly and uniformly 
in $ C^1( S \! \times \! S \backslash \{ x=y \}) $ 
by the assumption \eqref{:21d}, and $ |x_k-y_i|\ge 1/q$ by \eqref{:41c}, 
we deduce the claim $ \cref{;56b}   < \infty $ from \eqref{:41f}. 
\end{proof}

\begin{lem} \label{l:42} 
Let $ \Ct \label{;44c} =
\cref{;33}(\n _0 ) e^{ \{\supN |\mmr | \}  4 m b_r \cref{;44} }$. 
Then, for $\mukNm $-a.e.\ $\mathsf{s} $, it holds that 
\begin{align} \label{:42a}&\quad \quad 
\cref{;44c}^{-1} \le \st \ABN (\mathsf{x}) \le \cref{;44c} 
\quad \text{ for all $ \mathsf{x} \in\Srm ,\  r<s \in  \mathbb{N}   ,\ \text{and }\n _0 \le \n \in \mathbb{N}    $. }
\end{align}
\end{lem} 
\begin{proof} 
Since the diameter of $ \Sr $ is $ b_r $ and the number of the particles in $ \Sr $ is 
$ m $, we see that 
\begin{align*}&
 | \xrdash - \xr | \le 2 m b_r \quad \text{ for all } \mathsf{x}, \mathsf{x}' \in \Srm 
.\end{align*}
Hence we deduce from this, \eqref{:34a}, and \eqref{:41a} that 
\begin{align}\label{:42b}
\frac{\st \ABN  (\mathsf{x})}{\st \ABN  (\mathsf{x}')} 
 &= 
\frac
{e^{ - \mmr\cdot \xr - \wPsiNrs (\mathsf{x},\mathsf{s})} \tCNtilde }
{e^{ - \mmr\cdot \xrdash - \wPsiNrs (\mathsf{x}' ,\mathsf{s})} \tCNNtilde } 
% \\ \notag & 
\le e^{ |\mmr | 2 m b_r \cref{;44} } \frac{\tCNtilde }{\tCNNtilde } 
.\end{align}

We set $ \Xi = \Xi _{r}^{m}$ and $\hat{\Xi } = \hat{\Xi }_{r}^{m}$ by 
\begin{align*}&
\Xi = \{ (\n , s , \mathsf{x},\mathsf{x}')\, ;\, 
\n \in \mathbb{N},\ r<s \in\mathbb{N},\ \mathsf{x},\mathsf{x}' \in\Srm 
\} ,\\ \notag &
\hat{\Xi } = \{ (\n , s , t , \mathsf{x},\mathsf{x}')\, ;\, 
\n \in \mathbb{N},\ r<s < t \in\mathbb{N},\ \mathsf{x},\mathsf{x}' \in\Srm 
\} 
.\end{align*}
Then, by \eqref{:34c}, we have for $\mukNm $-a.e.\ $\mathsf{s} $ 
\begin{align}\label{:42c}&
%\supN \sups \sup_{\mathsf{x}, \mathsf{x}' }
\sup_{\Xi }
\{ \frac{ \tCNtilde }{\tCNNtilde }\}  
\\ \notag = & 
% \supN \sups \sup_{\mathsf{x}, \mathsf{x}' \in\Srm }
\sup_{\Xi }
 \{ 
\frac{\int _{\SSS }1_{\Hrk }(\pi _{\Srs }(\mathsf{s})+
\mathsf{z})
e^{ -\wPsiNrsi (\mathsf{x},\mathsf{z})- \wPsiNssi (\mathsf{s},\mathsf{z})}
\mukNm \circ \pi _{\Ssinfty }^{-1} (d\mathsf{z}) }
{\int _{\SSS }1_{\Hrk }(\pi _{\Srs }(\mathsf{s})+
\mathsf{z})
e^{ -\wPsiNrsi (\mathsf{x}',\mathsf{z})- \wPsiNssi (\mathsf{s},\mathsf{z})}
\mukNm \circ \pi _{\Ssinfty }^{-1} (d\mathsf{z})
} \} 
\\ \notag = & 
%\supN \sup_{r < s < t \in \mathbb{N}   } \sup_{\mathsf{x}, \mathsf{x}' \in\Srm } 
\sup_{\hat{\Xi } }
\{ 
\frac{\int _{\SSS }1_{\Hrk }(\pi _{\Srs }(\mathsf{s}) + \mathsf{z})
e^{ -\wPsiNrst (\mathsf{x},\mathsf{z})- \wPsiNrsst (\mathsf{s},\mathsf{z})}
\mukNm \circ \pi _{\Ssinfty }^{-1} (d\mathsf{z}) }
{\int _{\SSS }1_{\Hrk }(\pi _{\Srs }(\mathsf{s}) + \mathsf{z})
e^{ -\wPsiNrst (\mathsf{x}',\mathsf{z})- \wPsiNrsst (\mathsf{s},\mathsf{z})}
\mukNm \circ \pi _{\Ssinfty }^{-1} (d\mathsf{z})
} \} 
\\ \notag 
\le &\ e^{\cref{;44}}\quad \quad \quad \text{ by }\eqref{:41a} 
.\end{align}
Here we used $ \muN (\mathsf{s} ( S )\le \nN )= 1$ in the third line. 

Let $ \Ct \label{;45} = 2 \supN |\mmr | m b_r \cref{;44} $. 
Then \eqref{:42b} and \eqref{:42c} yield that 
\begin{align*} &
%\supN \sups \sup_{\mathsf{x}, \mathsf{x}' \in\Srm }
\sup_{\Xi }
\frac{ \st \ABN (\mathsf{x})} {\st \ABN (\mathsf{x}')} 
\le e^{ \cref{;45}}  \quad \text{ for $\mukNm $-a.e.\ $\mathsf{s} $}
.\end{align*}
Hence for $\mukNm $-a.e.\ $\mathsf{s} $, we see that for all 
 $ \mathsf{x}, \mathsf{x}' \in\Srm ,\  \ r<s \in  \mathbb{N}    ,\ \text{and } \n \in  \mathbb{N}   $, 
\begin{align} & \label{:42d} 
e^{-\cref{;45}} {\st \ABN (\mathsf{x}')}\le 
{\st \ABN (\mathsf{x})}\le 
e^{\cref{;45}} {\st \ABN (\mathsf{x}')}
.\end{align}
Multiply \eqref{:42d} by 
$ 1_{\Srm }(\mathsf{x}')e^{-\widetilde{\mathcal{H}}_{r}^{\n } (\mathsf{x}')}$ 
and integrate w.r.t.\ $ \Lambda (d\mathsf{x}')$. 
Note that by \eqref{:32c}, we have 
$ \int _{\Srm }{\st \ABN (\mathsf{x}')}e^{-\widetilde{\mathcal{H}}_{r}^{\n } (\mathsf{x}')} 
\Lambda (d\mathsf{x}') =1$. 
Then we deduce that for $\mukNm $-a.e.\ $\mathsf{s} $, 
\begin{align}&\notag 
e^{-\cref{;45}} 
\le {\st \ABN (\mathsf{x})} 
\int _{\Srm }e^{-\widetilde{\mathcal{H}}_{r}^{\n } (\mathsf{x}')} 
\Lambda (d\mathsf{x}') 
\le 
e^{\cref{;45}} 
\quad \text{ for all }\mathsf{x} \in\Srm 
.\end{align}
This combined with \lref{l:33} yields \eqref{:42a}. 
\end{proof}
%%%%%%%%%%%%%%%%%%%%%%%%%%
%%%%%%%%%%%%%%%%%%%

%
\begin{lem} \label{l:43} 
$ \mukNm \circ \pi _{ \Srs }^{-1}$ 
converges weakly to $ \murk \circ \pi _{ \Srs }^{-1} $ as 
$ \n \to\infty $. 
\end{lem}
\begin{proof} 
Let $ E $ be the discontinuity points of $ \pi _{ \Srs }$, namely 
$$ E = \{ \mathsf{s}\in\SSS \, ;\, 
\lim_{n\to\infty}\pi _{\Srs }(\mathsf{s}_n)
\not= \pi _{\Srs }(\mathsf{s}) \text{ for some }
\{ \mathsf{s}_n  \} \text{ such that }
\limi{n} \mathsf{s}_n =\mathsf{s}\}  
.$$
Then by \thetag{H.1}, we deduce that $ \murk (E)\le \mu (E)=0 $. 
Since $ \mukNm $ converge weakly to $  \murk $ by \lref{l:31} and 
the discontinuity points of 
$ \pi _{ \Srs }^{-1} $ are $  \murk $-measure zero, we obtain \lref{l:43}.  
\end{proof}

Let $\mathcal{H}_{rs} = \mathcal{H}_{\Srs }^{\Phi ,\Psi }$ and 
$ \widetilde{\mathcal{H}}_{rs}^{\n } $ 
such that 
\begin{align}\label{:44a} 
 \widetilde{\mathcal{H}}_{rs}^{\n } (\mathsf{x}) 
 &= 
 \sum_{x_i\in \Srs } \{\PhiN ( x_i )  - \mmi  \cdot x_i \} + 
 \sum_{x_i, x_j\in \Srs , i < j } \PsiN ( x_i, x_j)  
.\end{align}
By \eqref{:21a} and \eqref{:21b}, we see that 
$ \mukNm \circ \pi _{ \Srs }^{-1} $ and 
$ \murk \circ \pi _{ \Srs }^{-1} $ are absolutely continuous w.r.t.\ 
$ e^{- \widetilde{\mathcal{H}}_{rs}^{\n } }\Lambda $ and 
$ e^{-\mathcal{H}_{rs}}\Lambda $, respectively. 
Hence, we denote by $ \Delta ^{\n }   $ and $ \Delta $ 
their Radon-Nikodym densities, respectively. Namely, 
\begin{align}\label{:44b}&
 \Delta ^{\n } (\mathsf{s}) 
= \frac{\mukNm \circ \pi _{ \Srs }^{-1}(d\mathsf{s})}
{ e^{- \widetilde{\mathcal{H}}_{rs}^{\n } }\Lambda (d\mathsf{s})}
,\quad 
 \Delta (\mathsf{s}) 
= \frac{ \murk \circ \pi _{ \Srs }^{-1} (d\mathsf{s})}
{ e^{-\mathcal{H}_{rs}}\Lambda (d\mathsf{s})}
.\end{align}

%%%%[] ここから、変更　12/8/12
The following is the main result of this section. 
\begin{prop} \label{l:45}
$ \Delta ^{\n }  e^{- \widetilde{\mathcal{H}}_{rs}^{\n } } $ converges to 
$ \Delta e^{-\mathcal{H}_{rs}}$ in $ \LABone $ as $ \n \to\infty $. 
\end{prop}

We devote the rest of this section to the proof of \pref{l:45}. 
This proof is rather long, and we will complete it 
after preparing a sequence of lemmas. 

\begin{lem} \label{l:44}
\pref{l:45} follows from the relative compactness of 
$\{\Delta ^{\n } e^{- \widetilde{\mathcal{H}}_{rs}^{\n } }\}_{\n \in \mathbb{N}   }$ in $ \LABone $. 
\end{lem}
\begin{proof}
If $\{\Delta ^{\n } e^{- \widetilde{\mathcal{H}}_{rs}^{\n } }\}_{\n \in \mathbb{N}   }$ are 
relatively compact in $ \LABone $, then their limit points are unique 
and equal to $ \Delta e^{-\mathcal{H}_{rs}}$ by \lref{l:43}. 
\end{proof}

To prove the relative compactness as above, we use various kinds of cut-off 
procedures. 

Recall that $ \mathsf{S}^{n}_{rs} = \{\mathsf{x}\in \SSS ;\,\mathsf{x}(\Srs ) = n \}$. 
We set $ \DeltaN =  \Delta ^{\n }1_{\mathsf{S}^{n}_{rs}} $. Then we have 
\begin{align}\label{:46a}&
 \Delta ^{\n } =  %%e^{- \widetilde{\mathcal{H}}_{rs}^{\n } }  
\sum_{n=0}^{\infty} \DeltaN %% 1_{\mathsf{S}^{n}_{rs}} 
.\end{align}
We begin by considering a cut-off of $ \Delta ^{\n } e^{- \widetilde{\mathcal{H}}_{rs}^{\n } } $ 
according to \eqref{:46a}.  
\begin{lem} \label{l:46} 
For each $ \epsilon >0 $, there exists an $ n_0 $ such that 
\begin{align}\label{:46b}&
\supN 
\| \{  \sum_{n=n_0}^{\infty}\DeltaN \} 
  e^{- \widetilde{\mathcal{H}}_{rs}^{\n } } 
 \|_{ \LABone }
 < \epsilon 
.\end{align}
\end{lem}
\begin{proof}
By \lref{l:43}, we  see that  the sequence 
$ \{ \mukNm \circ \pi _{ \Srs }^{-1}\} $ is tight. 
Hence we deduce that for each $ \epsilon >0 $ 
there exists an $ n_0 $ such that 
\begin{align}\label{:46c}&
\supN  \mukNm  (\sum_{n=n_0}^{\infty} \mathsf{S}^{n}_{rs} )
 < \epsilon 
,\end{align}
which is equivalent to \eqref{:46b}.  
\end{proof}

According to \pref{l:45} and \eqref{:46b}, the relative compactness of 
$\{\Delta ^{\n } e^{- \widetilde{\mathcal{H}}_{rs}^{\n } }\}_{\n \in \mathbb{N}   }$ in $ \LABone $ 
follows from that of 
$ \{ \DeltaN e^{- \widetilde{\mathcal{H}}_{rs}^{\n } } \}_{\n \in \mathbb{N}   } $ 
for each $ n \in  \mathbb{N}   $. Hence, we fix $ n \in \mathbb{N}   $ in the rest of this section. 

Let $ \Hsl $ be as in \eqref{:41y}. 
We consider new sequences of cut-off measures 
$\{ \muNl   \}_{l\in\mathbb{N}}  $ such that 
\begin{align}\label{:47b}&
\muNl  = \mukNm (\cdot \cap \mathsf{S}^{n}_{rs}\cap \Hsl )
.\end{align}
Let $ \DeltaN _{l} $ be the Radon-Nikodym density of 
$ \muNl  \circ \pi _{ \Srs }^{-1}$ 
w.r.t.\ $ e^{- \widetilde{\mathcal{H}}_{rs}^{\n } }\Lambda $; that is, 
\begin{align}\label{:47c}&
 \DeltaN _{l} (\mathsf{s})=
 \frac{ \muNl  \circ \pi _{ \Srs }^{-1} (d\mathsf{s}) }
{e^{- \widetilde{\mathcal{H}}_{rs}^{\n } }\Lambda (d\mathsf{s}) }
.\end{align}

\begin{lem} \label{l:47} 
Let $ \DeltaN _{l} $ be as \eqref{:47c}. 
Then, for each $ n \in \mathbb{N }$, we have 
\begin{align}\label{:47d}&
\lim_{l\to\infty} \limsup_{\n \in \mathbb{N}   }
\| 
\DeltaN e^{- \widetilde{\mathcal{H}}_{rs}^{\n } }-
\DeltaN _{l}e^{- \widetilde{\mathcal{H}}_{rs}^{\n } }
\|_{\LABone } = 0 
.\end{align}
\end{lem}
\begin{proof}
Since $ \muNl  \le \mukNm $ by \eqref{:47b}, we see that 
$ 
\DeltaN _{l}e^{- \widetilde{\mathcal{H}}_{rs}^{\n } }\le 
\DeltaN e^{- \widetilde{\mathcal{H}}_{rs}^{\n } } 
$. 
This together with \eqref{:44b} and \eqref{:47c} yields 
%%%%%%%%%%%%%%%%%%%%%%%%%%%%%%%%
\begin{align}\label{:47e} &
\| 
\DeltaN e^{- \widetilde{\mathcal{H}}_{rs}^{\n } } - 
\DeltaN _{l}e^{- \widetilde{\mathcal{H}}_{rs}^{\n } }
\|_{\LABone } \le \mukNm (\Hsl ^c)
.\end{align}
From $ \mukNm \le \muN $ and \eqref{:21g} 
we deduce that 
\begin{align}\label{:47f}&
\lim_{l\to\infty}\limsup_{\n \in \mathbb{N}   } \mukNm (\Hsl ^c) \le 
\lim_{l\to\infty}\limsup_{\n \in \mathbb{N}   } \muN (\Hsl ^c) = 0 
.\end{align}
Combining \eqref{:47e} and \eqref{:47f} yields \eqref{:47d}. 
 \end{proof}
According to \lref{l:47}, it only remains to prove the relative compactness of 
 $\{\DeltaN _{l}e^{- \widetilde{\mathcal{H}}_{rs}^{\n } }\}_{\n \in \mathbb{N}   }$ 
in $ \LABone $ for all sufficiently large $ l \in \mathbb{N}   $. 
Hence, we fix such an $ l \in \mathbb{N}   $ in the rest of this section. 

Let $ B_r^{q} = \{ x ; |x-\Sr | < 1/q \} \backslash \Sr $; that is, 
$ B_r^{q} $ is the  intersection of $ \Sr ^c$ and the $ 1/q$-neighborhood of $ \Sr $.  
Let  $ \Aq  $ be the subset of $  \mathsf{S}^{n}_{rs}\cap \Hsl  $ 
 with no particles in $ B_r^{q}$. Namely, 
\begin{align}\label{:48a}&
\Aq = \{ \mathsf{s} \in \mathsf{S}^{n}_{rs}\cap \Hsl  \, ;\, \mathsf{s} ( B_r^{q} ) = 0 \} 
.\end{align}

\begin{lem} \label{l:48}
For each $ \epsilon >0$, there exists a $ q_0\in \mathbb{N}   $ 
such that, for all $ q \ge q_0 $,  
\begin{align}\label{:48b}&
\supN \| 
 \DeltaN _{l}e^{- \widetilde{\mathcal{H}}_{rs}^{\n } } - 
\DeltaN _{l}e^{- \widetilde{\mathcal{H}}_{rs}^{\n } }1_{\Aq }
 \|_{\LABone } 
 \le \epsilon 
.\end{align}
\end{lem}
\begin{proof}
By the definitions of $ \Aq $ and $  B_r^{q} $, and from the property of $ 1 $-correlation function we deduce that  
\begin{align*}&
\| 
 \DeltaN _{l}e^{- \widetilde{\mathcal{H}}_{rs}^{\n } } - 
\DeltaN _{l}e^{- \widetilde{\mathcal{H}}_{rs}^{\n } }1_{\Aq }
 \|_{\LABone } 
 \le \mukNm ((\Aq )^c )
 \le \int_{ B_r^{q}} \rNone (x) dx 
.\end{align*}
We deduce from \eqref{:21a}--\eqref{:21d} that, 
for each $ \epsilon >0$, there exists a $ q_0\in \mathbb{N}   $ such that 
\begin{align*}&
\supN \int_{ B_r^{q}} \rNone (x) dx 
  \le \epsilon \quad \text{ for all }  q \ge q_0 
.\end{align*}
Combining these two equations, we obtain \eqref{:48b}. 
\end{proof}

We will prove that 
$ \{  \DeltaN _{l} e^{- \widetilde{\mathcal{H}}_{rs}^{\n } } 1_{\Aq }  \}_{\n \in \mathbb{N}   }$ 
are relatively compact in $ \LABone $ 
for each $ n \in \mathbb{N}$ and for all sufficiently large $ l , q\in \mathbb{N}   $. 
For this, we will prove both of the relative compactness of 
$ \{  e^{- \widetilde{\mathcal{H}}_{rs}^{\n } } 1_{\Aq }  \}_{\n \in \mathbb{N}   }$ in $ \LABone $, 
and that of 
$ \{  \DeltaN _{l} \}_{\n \in \mathbb{N}   }$ in  $ C_b(\Aq )$ with uniform norm 
$ \| \, \cdot \, \|_{ C_b(\Aq )} $, where 
$ \| \, f \, \|_{ C_b(\Aq )} = \sup\{ |f(\mathsf{y})|\, ;\, \mathsf{y} \in \Aq \}$. 

We begin by proving the first claim. 
\begin{lem} \label{l:49} 
$ \{  e^{- \widetilde{\mathcal{H}}_{rs}^{\n } } 1_{\Aq }  \}_{\n \in \mathbb{N}   }$ converge 
to $ e^{-\mathcal{H}_{rs}}1_{\Aq } $ in $ \LABone $, and 
\begin{align}\label{:49a}&
\inf_{\n \in\mathbb{N}} %\limi{\n } 
\| e^{- \widetilde{\mathcal{H}}_{rs}^{\n } }1_{\Aq } \|_{\LABone }  
% = \| e^{- \widetilde{\mathcal{H}}_{rs}^{\n } }1_{\Aq } \|_{\LABone }
 > 0 
\quad \text{ for all sufficiently large $ l , q\in \mathbb{N}   $}
. \end{align}
\end{lem}
\begin{proof} 
From \eqref{:21c} and \eqref{:21d} together with \eqref{:44a}, we deduce that 
\begin{align} \label{:49b}&
\limi{\n } e^{- \widetilde{\mathcal{H}}_{rs}^{\n } (\mathsf{x})} 
1_{\Aq } (\mathsf{x}) 
 = 
 e^{-\mathcal{H}_{rs}(\mathsf{x})}1_{\Aq } (\mathsf{x}) 
 \quad \text{ for $ \Lambda $-a.e.\ $ \mathsf{x}$}
,\\ \label{:49c}&
 \sup_{\n \in \mathbb{N}   } \sup_{\mathsf{x}\in\SSS } \{ 
e^{- \widetilde{\mathcal{H}}_{rs}^{\n } (\mathsf{x})} 1_{\Aq } (\mathsf{x}) \} < \infty 
.\end{align}
From \eqref{:49b} and  \eqref{:49c}, 
combined with the Lebesgue convergence theorem, we deduce the first claim. 
In turn, we deduce that 
\begin{align}\label{:49d}&
\limi{\n }  \| e^{- \widetilde{\mathcal{H}}_{rs}^{\n } }1_{\Aq } \|_{\LABone }  
= \| e^{-\mathcal{H}_{rs}}1_{\Aq } \|_{\LABone } 
. \end{align}
From \eqref{:21c} and \eqref{:21d}, we have, for all sufficiently large $ l , q\in \mathbb{N}   $, 
\begin{align}\label{:49e}& 
\| e^{- \widetilde{\mathcal{H}}_{rs}^{\n } }1_{\Aq } \|_{\LABone }  > 0 
\quad (\forall \n \in\mathbb{N}), \quad 
\| e^{-\mathcal{H}_{rs}}1_{\Aq } \|_{\LABone } > 0 
.\end{align}
Combining \eqref{:49d}  and \eqref{:49e} yields \eqref{:49a}.  
\end{proof}

We next prove the second claim. 
\begin{lem} \label{l:4(}
$ \{  \DeltaN _{l} \}_{\n \in \mathbb{N}   }$ are relatively compact 
 in  $ C_b(\Aq )$ with uniform norm. 
%in the uniform topology on $ \Aq $
\end{lem}
\begin{proof}
From the definition of $ \DeltaN _{l} $ 
(see \eqref{:30a}, \eqref{:47b}, and \eqref{:47c}), we see that 
\begin{align}\label{:4(a}&
\|  \DeltaN _{l} e^{- \widetilde{\mathcal{H}}_{rs}^{\n } } 1_{\Aq } \|_{ \LABone } 
= 
 \muNl  \circ \pi _{ \Srs }^{-1} (\Aq )
\le 1 
.\end{align}

Note that $ \pi _{\Srs ^c} = \pi _{\Sr } + \pi _{\Ssinfty }$. 
Hence we write $  \pi _{\Srs ^c} (\mathsf{s}) =  \mathsf{x} + \mathsf{z}$, 
where $ \mathsf{x}\in \pi _{\Sr }(\SSS )$ and $ \mathsf{z}\in \pi _{\Ssinfty }(\SSS )$. 
With this notation, $ \DeltaN _{l}(\mathsf{y}) $ can be written as 
\begin{align*}&
\DeltaN _{l}(\mathsf{y}) = \cref{;45f} 
\int _{\SSS } 
1_{\Hrk \cap \Hsl }( \mathsf{x} + \pi _{\Srs }(\mathsf{y})+\mathsf{z})
e^{
-\wPsiNrs (\mathsf{x},\mathsf{y})
- \wPsiNssi (\mathsf{y},\mathsf{z}) } 
\muNl  
\circ \pi _{\Srs ^c}^{-1}
(d\mathsf{x}d\mathsf{z})
.\end{align*}
with positive constant $ \Ct \label{;45f}$. 
Let 
$  \Ct \label{;45l} = \sup \{ 
e^{(\cref{;56b}   +l) d_{\mathsf{S}^{n}_{rs}}(\mathsf{y},\mathsf{y}')} ; 
\mathsf{y},\mathsf{y}'\in \Aq  \}$. 
Then applying  \eqref{:41d} and \eqref{:41b} to 
$ \wPsiNrs (\mathsf{x},\mathsf{y}) $ and $  \wPsiNssi (\mathsf{y},\mathsf{z}) $
 respectively, we deduce from \lref{l:41} that 
\begin{align}\label{:4(b}& 
\supN \sup_{ \mathsf{y},\mathsf{y}'\in \Aq } 
 \{ 
\frac{\DeltaN _{l}(\mathsf{y})}{\DeltaN _{l}(\mathsf{y}')} 
\} 
\le  \sup_{ \mathsf{y},\mathsf{y}'\in \Aq } 
e^{(\cref{;56b}   +l) d_{\mathsf{S}^{n}_{rs}}(\mathsf{y},\mathsf{y}')} 
= \cref{;45l} 
< \infty 
.\end{align}
Hence from \eqref{:4(a} and \eqref{:4(b}, we see that 
\begin{align}\label{:4(c} & &&
\| \DeltaN _{l}  \|_{ C_b(\Aq )}  \cdot 
 \| e^{- \widetilde{\mathcal{H}}_{rs}^{\n } }1_{\Aq } \|_{\LABone } 
&&
\\ \notag  &&
 = \ & 
 \| (
\frac{ \| \DeltaN _{l}  \|_{ C_b(\Aq )}  } 
{\DeltaN _{l}  } )
\DeltaN _{l} e^{- \widetilde{\mathcal{H}}_{rs}^{\n } }1_{\Aq } \|_{\LABone }
&\\ \notag  &&
 \le \ & \cref{;45l} 
 \|  \DeltaN _{l} e^{- \widetilde{\mathcal{H}}_{rs}^{\n } }1_{\Aq } \|_{\LABone }
&&
\text{ by \eqref{:4(b}}
\\ \notag  &&\le \ & 
 \cref{;45l} 
 &&\text{ by \eqref{:4(a}}
.\end{align}
Combining \eqref{:49a} and \eqref{:4(c} yields 
\begin{align}\label{:4(d}&
\sup_{\n \in\mathbb{N}} \| \DeltaN _{l}  \|_{ C_b(\Aq )}  
\le 
\frac{\cref{;45l}}
{ \inf_{\n \in\mathbb{N}} 
 \| e^{- \widetilde{\mathcal{H}}_{rs}^{\n } }1_{\Aq } \|_{\LABone } }
<  \infty 
.\end{align}

Taking the logarithm of \eqref{:4(b} and interchanging 
the role of $ \mathsf{y}$ and $ \mathsf{y}'$, we see that,
for all $  \mathsf{y},\mathsf{y}'\in \Aq   $, 
\begin{align}\label{:4(e}&
\supN 
 \{ |
\log {\DeltaN _{l}(\mathsf{y})} - \log {\DeltaN _{l}(\mathsf{y}')} | \}
 \le 
{(\cref{;56b}   +l) 
d_{\mathsf{S}^{n}_{rs}}(\mathsf{y},\mathsf{y}')}
.\end{align}
Then we deduce from the inequality 
$$ | x - y | \le  \max \{ x,y \} |\log x - \log y | 
\quad \text{ for } x,y > 0 
$$ 
and \eqref{:4(e} that, 
for all $  \mathsf{y},\mathsf{y}'\in \Aq $, 
\begin{align}\label{:4(f}  
& \supN 
 \{ |
 { \DeltaN _{l}(\mathsf{y})} - {\DeltaN _{l}(\mathsf{y}')} | \}
\\ \notag 
\le &  % e^{\cref{;4(f}}
\supN 
 \{  \| \DeltaN _{l}  \|_{ C_b(\Aq )} 
| \log {\DeltaN _{l}(\mathsf{y})} - \log {\DeltaN _{l}(\mathsf{y}')} | \}
\\ \notag 
 \le & 
\cref{;4(f} (\cref{;56b}   +l) d_{\mathsf{S}^{n}_{rs}}(\mathsf{y},\mathsf{y}') 
,\end{align}
where we set 
$ \Ct \label{;4(f} = \sup_{\n \in\mathbb{N}} \| \DeltaN _{l}  \|_{ C_b(\Aq )} $. 
Since $ \cref{;4(f} < \infty $ by \eqref{:4(d}, we deduce from 
\eqref{:4(f} that $ \{\DeltaN _{l} \}_{\n \in \mathbb{N}   } $ are equi-continuous 
in $ C_b(\Aq ) $ for each $ q\in \mathbb{N}   $,

From \eqref{:4(d} and \eqref{:4(f}, we deduce that 
$ \{\DeltaN _{l} \}_{\n \in \mathbb{N}   } $ are equi-continuous and 
 uniformly bounded in $ C_b(\Aq ) $ for each $ q\in \mathbb{N}   $. 
Hence, applying the Ascoli-Arzel\'{a} theorem to $ \{ \DeltaN _{l} \} $ 
completes the proof of \lref{l:4(}. 
\end{proof}

\bigskip

We are now in a position to complete the proof of \pref{l:45}. 

\smallskip 

\noindent 
{\em Proof of \pref{l:45}. } 
From \lref{l:49} and \lref{l:4(}, we deduce that  
$ \{  e^{- \widetilde{\mathcal{H}}_{rs}^{\n } } 1_{\Aq }  \}_{\n \in \mathbb{N}   }$ are 
convergent sequences in $ \LABone $ and that 
$\{  \DeltaN _{l} \}_{\n \in \mathbb{N}    } $ are relatively compact  in  $ C_b(\Aq )$ 
 for each $ n \in\mathbb{N}$ and for all sufficiently large $ l , q\in \mathbb{N}   $. 
Then we conclude that 
$$\{ \DeltaN _{l}e^{- \widetilde{\mathcal{H}}_{rs}^{\n } } 1_{\Aq } \}_{\n \in \mathbb{N}   }$$
are relatively compact in $ \LABone $ for such $ n , l ,q \in\mathbb{N}$. 
Combining this with \lref{l:46}, \lref{l:47}, and \lref{l:48}, we see that 
$\{\Delta ^{\n }e^{- \widetilde{\mathcal{H}}_{rs}^{\n } }\}_{\n \in \mathbb{N}   }$  
are relatively compact in $ \LABone $. 
Hence by \lref{l:44}, we complete the proof of \pref{l:45}. 
\qed 

\section{Proof of the second step. } \label{s:5} 
We devote this section to the proof of the second step. 

Let $ \murk = \mu ( \cdot \cap \mathsf{S}_{r,k}^{m} ) $ as in \dref{dfn:1}. 
Let $ \muABs $ be the regular conditional probability defined by 
$$ \muABs  = \murk (\pi _{ \Sr }(\mathsf{s})
\in d \mathsf{x} |\, \pi _{ \Srs }(\mathsf{s}))
.$$
We begin by proving uniform upper and lower bounds of 
Radon-Nikodym densities of 
$ \muABs $ w.r.t.\ $ e ^{- \mathcal{H}_{r}(\mathsf{x})}\mA $. 
\begin{lem} \label{l:51}
\thetag{1} 
For $ \murk $-a.e.\ $ \mathsf{s} $, 
the regular conditional probability $ \muABs $ 
is absolutely continuous w.r.t.\ 
$ e ^{- \mathcal{H}_{r}(\mathsf{x})}\mA $. 
\\
\thetag{2} 
Let $ \skABs $ be the Radon-Nikodym densities of 
$  \muABs $ w.r.t.\ $ e ^{- \mathcal{H}_{r}(\mathsf{x})}\mA $. 
Then, for each $ r,s, m,k \in  \mathbb{N}   $ such that $ r<s $ and 
$ \murk $-a.e.\ $ \mathsf{s} $, %the densities $ \skABs $ satisfy 
\begin{align} \label{:51a}&
\cref{;44c}^{-1}\le \skABs (\mathsf{x})\le \cref{;44c}
\quad \text{ for $ \muABs $-a.e.\ $ \mathsf{x}$} 
.\end{align}
Here $ \cref{;44c}$ is the positive constant given in \lref{l:42}. 
\end{lem}
\begin{proof} 
We first prove the claim \thetag{1}. 
Similar to the case of \lref{l:43}, we see that 
$ \mukNm \circ (\pi _{\Sr },\pi _{\Srs })^{-1} $ 
converge weakly to 
$ \murk \circ (\pi _{\Sr },\pi _{\Srs })^{-1} $ 
as $ \n \to\infty $. Hence, for 
$ \mathsf{f}$, $ \mathsf{g} \in C_b(\SSS ) $, 
we have 
\begin{align} \label{:51f} 
\int _{\SSS } 
\mathsf{f}(\pi _{\Sr }(\mathsf{s}))
\mathsf{g}(\pi _{\Srs }(\mathsf{s})) d\murk 
& = \limi{\n }
\int _{\SSS } 
\mathsf{f}(\pi _{\Sr }(\mathsf{s}))
\mathsf{g}(\pi _{\Srs }(\mathsf{s})) d\mukNm 
.\end{align}

By \lref{l:42} and the diagonal argument, 
there exist subsequences of $\{ \st \ABN \}_{\n \in \mathbb{N}} $, 
denoted by the same symbol, with a limit $\skABs $ such that, 
for all $ k,m, r<s\in  \mathbb{N}   $,  
\begin{align} \label{:51g}&
\limi{\n } 
 \st \ABN  (\pi _{ \Sr }(\mathsf{s})) = 
\skABs (\pi _{ \Sr }(\mathsf{s})) 
\quad \text{$ * $-weakly in $\LAB $}
.\end{align}
Here $ \skABs $ is a function such that $ \skABs (\mathsf{x}) = \skABss (\pi _{ \Sr }(\mathsf{x}))$. 
Let 
\begin{align}&\label{:51h}
\mathsf{F}^{\n }(\mathsf{s}) = 
\mathsf{f}(\pi _{\Sr }(\mathsf{s}))
\mathsf{g}(\pi _{\Srs }(\mathsf{s})) 
\Delta ^{\n } (\mathsf{s}) 
e^{-\widetilde{\mathcal{H}}_{r}^{\n } (\mathsf{s})} 
,\\&\label{:51i}
\mathsf{F}(\mathsf{s}) = 
\mathsf{f}(\pi _{\Sr }(\mathsf{s}))
\mathsf{g}(\pi _{\Srs }(\mathsf{s})) 
\Delta (\mathsf{s})
e^{-\mathcal{H}_{r}(\mathsf{s})} 
.\end{align}
Then by \pref{l:45}, we see that $ \mathsf{F}^{\n } $ converge to 
$ \mathsf{F} $ in $ \LABone $. This combined with \eqref{:51g} implies 
\begin{align} \label{:51j} & 
\limi{\n }\int _{\SSS } \mathsf{F}^{\n }(\mathsf{s}) 
 \st \ABN (\mathsf{s})d\Lambda 
 = 
\int _{\SSS }\mathsf{F}(\mathsf{s})
\skABs (\mathsf{s}) d\Lambda 
.\end{align}

By \eqref{:51f}, \eqref{:51j} and 
$ \Delta (\mathsf{y}) e^{-\mathcal{H}_{r}(\mathsf{y})} \Lambda (d\mathsf{y}) = 
\murk \circ \pi _{ \Srs }^{-1}(d\mathsf{y}) $, 
 we obtain 
\begin{align*}&
\int _{\SSS } 
\mathsf{f}(\mathsf{x})
\mathsf{g}(\mathsf{y})
d\murk 
 = \int _{\SSS } 
\mathsf{f}(\mathsf{x} )
\mathsf{g}(\mathsf{y})
\skABs (\mathsf{x}) e^{-\mathcal{H}_{r}(\mathsf{x})} 
\mA \murk \circ \pi _{ \Srs }^{-1}(d\mathsf{y}) 
,\end{align*}
where $ \mathsf{x} = \pi _{\Sr }(\mathsf{s})$ and 
$ \mathsf{y} = \pi _{\Srs }(\mathsf{s})$. 
Hence, we obtain \thetag{1} with density $ \skABs $. 

By \eqref{:42a} and \eqref{:51g}, 
we see that $\skABs $ satisfies \eqref{:51a}, which implies \thetag{2}. 
\end{proof}

\begin{lem} \label{l:52} 
Let $ \muA (d\mathsf{x}) $ be as in \eqref{:qg3}. 
Let $ \skABs $ be as in \lref{l:51}. Then the following limit exists. 
\begin{align}\label{:52a}& 
\sAym (\mathsf{x}) : = \limi{s} \skABs (\mathsf{x}) 
\quad \text{for $ \muA $-a.s.\ $\mathsf{x}$, 
 for $ \murk $-a.s.\ $ \mathsf{s} $}
.\end{align}
Moreover, $ \sAym $ satisfies for $ \murk $-a.e.\ $ \mathsf{s} $ 
\begin{align}\label{:52b}&
\cref{;44c}^{-1}\le \sAym (\mathsf{x})\le \cref{;44c}
\quad \text{ for $ \muA $-a.e.\ }\mathsf{x} 
\\ \label{:52c}&
\sAym (\mathsf{x}) e^{-\mathcal{H}_{r}(\mathsf{x})}
\mA = \muA (d\mathsf{x})
. \end{align}
\end{lem}
\begin{proof}
The proof of this lemma is exactly the same as Lemma 5.5 in \cite{o.rm}. 
However, we give the proof here for the reader's convenience.

Define $ \map{M_s}{\mathsf{S}}{ \mathbb{R} } $ by 
$ M_s (\mathsf{s}) = \skABs (\mathsf{x})$, where 
$ \mathsf{x} = \pi _{\Sr }(\mathsf{s})$.
Recall that 
$ \skABs $ is the Radon-Nikodym density of 
$ \muABs $ w.r.t.\ $ e ^{- \mathcal{H}_{r}(\mathsf{x})} \mA $ and that 
$ \muABs = \mu ^{m}_{r,k,\pi _{\Srs }(\mathsf{s}),rs}$ by construction. 
Hence, 
\begin{align} \label{:52d}
M_s (\mathsf{s})e^{-\mathcal{H}_{r}(\mathsf{x}) }\mA 
& = 
\mu ^{m}_{r,k,\pi _{\Srs }(\mathsf{s}),rs}
(d\mathsf{x})
.\end{align}

Let 
$ \mathcal{F}_s = 
\sigma [\pi _{\Sr }, \pi _{\Srs }] $, where 
$ r< s \le \infty $. 
Then by \eqref{:52d}, we see that 
$ \{ M_s \}_{s \in [r, \infty)}$ 
is an $ (\mathcal{F}_s) $-martingale, which implies 
$ M_{\infty }(\mathsf{s}): = 
\limi{s} M_s (\mathsf{s})$ 
exists for $ \murk $-a.e.\ $\mathsf{s}$. 
Since 
\begin{align*}&
M_s (\mathsf{s}) = 
\sigma ^m_{r,k,\pi_{\Srs}(\mathsf{s}),rs}
(\mathsf{x}), \quad \text{where }
 \mathsf{x} = \pi _{\Sr }(\mathsf{s}) 
,\end{align*}
we write 
$ M_{\infty }(\mathsf{s}) =  \sAym (\mathsf{x}) $.  
By construction, 
$  \sAym (\mathsf{x}) = 
\sigma _{r,k,\pi_{\Srinfty }(\mathsf{s})}^{ m } (\mathsf{x})  =
\sigma _{r,k,\pi_{\Sr ^c}(\mathsf{s})}^{ m } (\mathsf{x})  
$ and, for $ \murk $-a.s.\ $ \mathsf{s} $, we can regard 
$  \sAym (\mathsf{x})$ as a  $ \sigma [ \pi _{\Sr }]$-measurable function in $ \mathsf{x}$. 
Hence, through the disintegration \eqref{:qg4}, we obtain \eqref{:52a}. 

We immediately obtain \eqref{:52b} 
from \eqref{:51a} and \eqref{:52a}. 

We see that 
$ \{ M_s \}_{s \in [r, \infty)}$ is uniformly integrable by \eqref{:51a}. 
Hence, by \eqref{:52a}, we see that $ M_s (\mathsf{s})$ converges to 
$ M_{\infty }(\mathsf{s}) = \sAym (\mathsf{x})$ strongly in $ L^1(\Srm , \muA ) $, 
which combined with \eqref{:52d} and the definition 
$ M_s (\mathsf{s}) = \skABs (\mathsf{x}) $ yields \eqref{:52c}. 
\end{proof}

\noindent
\textit{Proof of \tref{l:21}. } 
By \lref{l:31}, we see that $ \{ \murk \}$ satisfies \eqref{:qg1}. 
Moreover, by \eqref{:52b} and \eqref{:52c}, we deduce 
that $ \murky $ satisfies \eqref{:qg2}, 
which completes the proof of \tref{l:21}. 
\qed

%%%%%%%%%%%%%%%%%%%% ここから %Section 6
\section{A sufficient condition of \thetag{H.3}} \label{s:6}
% \footnote{以下は、$ \mmi =0$の場合を想定している。}??
In this section, we give a sufficient condition of \thetag{H.3} 
when $ d= 1,2$ and $ \PsiN $ satisfy \eqref{:22a}. 
So $ \PsiN (x,y) := \Psi (x,y) = - \beta \log |x -y |$ are logarithmic functions by assumption.  
If $ d=2$, we regard $  \mathbb{R} ^2$ as $ \mathbb{C}$ by the natural correspondence: 
$ \mathbb{R}^2\ni (x,y) \mapsto x + \sqrt{-1}y \in \mathbb{C}$. 
 To unify the both cases we regard $ \mathbb{R}$ as a subset of $ \mathbb{C}$ 
 in an obvious manner. 
We denote by $ \Re [\cdot ]$ and $ \Im [\cdot ]$ the real and imaginary part of 
$ \cdot $, respectively. 
We remark that $ z/|z|^2 = 1/\bar{z}\in\mathbb{C}$.

We consider the Taylor expansion of $ \Psi ( x , y ) $. %
\begin{lem} \label{l:61} 
Assume \eqref{:22a}.  Let $ x,y\in \mathbb{C} $ 
such that $ | x |<| y | $. Then 
\begin{align}\label{:61a}&
\Psi ( x , y ) -\Psi (0, y )  
 = 
\beta \sum_{\ell = 1}^{\infty} \frac{1}{\ell }
\Re [(\frac{ x }{ y })^{\ell}]
.\end{align}
Here $ \Re [{\cdot }] $ denotes 
the real part of $ \cdot \in \mathbb{C}$. 
\end{lem}
\begin{proof} 
Let $ r = | x |/| y | $ and $ \theta = \angle ( x , y ) $. 
Then we see that 
\begin{align} \notag 
\Psi ( x , y )  & -\Psi (0, y ) = -\frac{\beta }{2} 
\log | \frac{ x }{| y |}- \frac{ y }{| y |}|^2 
\\ \notag & 
= -\frac{\beta }{2} \log \left(1+r^2- 2r \cos \theta \right)  = 
-\frac{\beta }{2} \{\log (1-re^{ \mathrm{i}\theta }) + \log (1-re^{- \mathrm{i} \theta })\}
.\end{align}
Hence, \eqref{:61a} follows from the Taylor expansion. 
\end{proof}%%%%

Let $ \Srs = \Ss \backslash \Sr = \{ y\in  S \, ;\, \br \le |y|< b_s \} $ 
as before, where $ \Sr $ and $ \br $ are given by \eqref{:qg0}.  We set 
 $ \Psi _{rs} (x,\mathsf{y}) = \sum _{y_i\in \Srs }  \Psi (x,y_i) $, 
where $ \mathsf{y}=\sum_i\delta_{y_i}$.  
By \eqref{:61a}, we easily see that 
\begin{align} \label{:63h} 
\Psi _{rs} (x,\mathsf{y}) - \Psi _{rs} (w ,\mathsf{y}) & = 
\beta \sum_{\ell = 1}^{\infty } \frac{1}{\ell } \4 
\Re [\frac{ x ^{\ell }  -  w ^{\ell }  }{ y _i^{\ell }  }]
\\ \notag &  
= 
\beta \sum_{\ell = 1}^{\infty } \frac{1}{\ell } 
\Re [( x ^{\ell }  -  w ^{\ell }  )\cdot \4 \frac{1}{ y _i^{\ell }  }]
.\end{align}

Recall the notation $  x \cdot \mms $ in \eqref{:21c}.  
If $ d=2$, then 
$ \mms = (\mathfrak{m}^{\n }_{s,1}, \mathfrak{m}^{\n }_{s,2}) \in \mathbb{R}^2$ by definition, 
and so 
$  x \cdot \mms  = 
x_1 \mathfrak{m}^{\n }_{s,1} + x_2 \mathfrak{m}^{\n }_{s,2} $. 
Since we interpret $ x $ as complex numbers, we set 
$  x \cdot \mms  = 
\Re [ x ] \mathfrak{m}^{\n }_{s,1}+ \Im [ x ] \mathfrak{m}^{\n }_{s,2} 
= \Re [ x \bar{\mathsf{m}}_{s}^{\n } ]$.  
Since $ x = x_1 + \sqrt{-1}x_2 $,  we then have 
\begin{align*} 
x \cdot (\mmr - \mms )  & = 
 x _1 (\mathfrak{m}^{\n }_{r,1} - \mathfrak{m}^{\n }_{s,1}) 
+ 
 x _2 (\mathfrak{m}^{\n }_{r,2} - \mathfrak{m}^{\n }_{s,2}) 
\\ \notag & = 
\Re [x ( \bar{\mathsf{m}}_{r}^{\n } -  \bar{\mathsf{m}}_{s}^{\n } ) ]
.\end{align*}
Here, in the second line, we regard $ x $, $ \mmr $, and $ \mms $ 
as complex numbers.

\begin{lem} \label{l:62} 
Let 
\begin{align}  \label{:63i} 
\widetilde{\Psi }^{\n }_{rs} (x,\mathsf{y})
= &
\Psi _{rs} (x,\mathsf{y})  + 
\Re [x ( \bar{\mathsf{m}}_{r}^{\n } -  \bar{\mathsf{m}}_{s}^{\n } ) ]
.\end{align}
Then %for each $ x,w \in \Sr $ such that $ x\not= w $ 
the following holds with finite constants $ \Ct \label{;6x} $ and $ \Ct \label{;6y} $. 
\begin{align} 
\label{:62d} & 
\sup_{x,w \in \Sr \atop  x\not= w }
\frac
{|\widetilde{\Psi }^{\n }_{rs} (x,\mathsf{y}) -  \widetilde{\Psi }^{\n }_{rs} (w ,\mathsf{y}) |}
{| x -  w |}
\\ \notag & 
\le  
 |\mmmm | 
 + 
\cref{;6x} 
\sum_{\ell = 2}^{\ellell -1} |\9 |
+ 
\cref{;6y} 
\sum_{\y _i \in \Srs } 
\frac{ \br ^{\ellell } }{| y _i |^{\ellell } - \br ^{\ellell }  }
.\end{align}
Here $ \mathsf{y}= \sum_j \delta_{y_j}$ as before and 
\begin{align} \label{:63k}
 \mmmm & =  \Re [ 
\beta \big( \4 \frac{1}{ y _i } \big)  +  ( \bar{\mathsf{m}}_{r}^{\n } -  \bar{\mathsf{m}}_{s}^{\n } ) ]  
, \\ \label{:63c}
\cref{;6x} & =  |\beta | \cdot \max_{1 \le \ell < \ellell } 
 \sup _{x\not=w\in\Sr } 
\frac{| x ^{\ell } -  w ^{\ell } |}{\ell | x -  w |} ,
\\\label{:62e}%\quad 
\cref{;6y} & =  |\beta | \cdot \sup_{\ellell \le \ell } \sup _{x\not=w\in\Sr }
\frac{| x ^{\ell } -  w ^{\ell } |}{ \br ^{\ell }  \ell | x -  w |}
.\end{align}
\end{lem}

\begin{proof}
We first check the finiteness of $ \cref{;6x}$ and $ \cref{;6y}$. 
Indeed, $ \cref{;6x} < \infty $ is clear. 
Note that, $ | x |/ \br < 1$ on $ \Sr $. Thus, the Lipschitz norm of the function 
$ { x ^{\ell } }/ {\br ^{\ell }  \ell }$ on $ \Sr $ is uniformly bounded in $ \ell $, 
which implies $ \cref{;6y} < \infty $. 

Since $ |\Re [ab]|\le |a||b|$, we deduce from \eqref{:63h} that for 
$ x\not= w \in \Sr $ 
\begin{align}\label{:62z}&
\frac
 {|\widetilde{\Psi }^{\n }_{rs} (x,\mathsf{y}) -  \widetilde{\Psi }^{\n }_{rs} (w ,\mathsf{y}) |} 
 {| x -  w |}
 \le 
 |\mmmm |  
 + 
 |\beta | \{ \sum_{\ell = 2}^{\infty }\frac{| x ^{\ell } -  w ^{\ell } |}{\ell | x -  w |} \} 
 \big| \4 \frac{1}{ y _i^{\ell } }\big| 
.\end{align}
We easily see that 
\begin{align}\label{:62zz} 
&  
 |\beta | \{ \sum_{\ell = 2}^{\infty }\frac{| x ^{\ell } -  w ^{\ell } |}{\ell | x -  w |} \} 
 \big| \4 \frac{1}{ y _i^{\ell } }\big| 
\\ \notag 
= & 
 |\beta | \{ \sum_{\ell = 2}^{\ellell -1} \frac{| x ^{\ell } -  w ^{\ell } |}{\ell | x -  w |} \} 
 \big| \4 \frac{1}{ y _i^{\ell } }\big| 
 + 
 |\beta | \{ \sum_{\ell = \ellell }^{\infty} \frac{| x ^{\ell } -  w ^{\ell } |}{\ell | x -  w |} \} 
 \big| \4 \frac{1}{ y _i^{\ell } }\big| 
,\\ \notag  
\le &  
 \cref{;6x} 
\sum_{\ell = 2}^{\ellell -1} 
|\9 |
+ 
\cref{;6y} 
\sum_{\ell = \ellell }^{\infty } 
\sum_{\y _i \in \Srs }\frac{ \br ^{\ell } }{| y _i |^{\ell }} 
\\ \notag 
= & \cref{;6x} 
\sum_{\ell = 2}^{\ellell -1} |\9 |
+ 
\cref{;6y} 
\sum_{\y _i \in \Srs } 
\frac{ \br ^{\ellell } }{| y _i |^{\ellell } - \br ^{\ellell }  }
.\end{align}
Here we used the formula 
$ \sum_{\ell = \ellell }^{\infty } { a^{\ell } }/{b^{\ell }} = 
 {a^{\ellell }}/({b^{\ellell } - a^{\ellell }  })$ valid for $ 0< a\le b $ in the last line. 
If $ a=b$, then we interpret $ \sum_{\ell = \ellell }^{\infty } { a^{\ell } }/{b^{\ell }} = \infty $. 

Combining \eqref{:62z} and \eqref{:62zz} completes the proof of 
\lref{l:62}.  
\end{proof}

Taking \eqref{:63i}, \eqref{:62d}, and \eqref{:63k}  into account, 
we set for $ r , \ak ,  \ell \in \mathbb{N} $ 
\begin{align} \label{:63x} &
\mathsf{U}_{r,1,k} = \{ \mathsf{y}\in \mathsf{S} \, ;\, \sups \supN |\mmmm |  \le \ak \} 
,\\ \label{:63y} &
\mathsf{U}_{r,\ell ,k} = \{ \mathsf{y}\in \mathsf{S} \, ;\,  
 \sups |\9 | \le \ak \} \quad \text{if }2 \le \ell 
,\\
\label{:63z} &
\bar{\mathsf{U}}_{r,\ell ,k} = \{ \mathsf{y}\in \mathsf{S} \, ;\, 
 \{ \sum_{\y _i \in \Sri } \frac{1}{| y _i |^{\ell }- \br ^{\ell } } \} 
\le \ak \} 
.\end{align}
We introduce the new condition \thetag{H.5}. 

\medskip
\noindent
\thetag{H.5} For each $ r \in \mathbb{N}   $, there exists 
an $ \ellell$  such that $ 2 \le \ellell \in  \mathbb{N}   $ and that 
\begin{align} 
\label{:63p} &
\limi{k} \limsup_{ \n \to \infty } 
\muN (\bar{\mathsf{U}}_{r,\ellell ,k} ^c ) = 0 
,\\ 
\label{:63q} & 
\limi{k} \limsup_{ \n \to \infty } \muN (\mathsf{U}_{r,\ell ,k} ^c ) = 0 
\quad \text{ for all $  1 \le \ell < \ellell $}
.\end{align}

We now state the main theorem of this section. 
\begin{thm}	\label{l:63}
Assume \eqref{:22a} and $  S = \mathbb{C}$. 
Then \thetag{H.5} implies \thetag{H.3}. 
\end{thm}
\begin{rem}\label{r:63}
If $ d=1 $, then 
\begin{align}\label{:63a}&
|\mmmm |   =  | \mathsf{v}_{1,rs}^{\n } (\mathsf{y} ) | 
.\end{align}
Hence \tref{l:63} is also valid for the proof of \tref{l:22}. 
In fact, we see that 
\begin{align}\label{:63b}&
\mathsf{U}_{r,1,k} = \{ \mathsf{y}\in \mathsf{S} \, ;\,  
 \sups \supN   | \mathsf{v}_{1,rs}^{\n } (\mathsf{y} ) |  \le \ak \}
.\end{align}
\end{rem}
\begin{proof} 
Set $ \Ct \label{;6z} = \cref{;6y} \br ^{\ellell } $. 
Then from \eqref{:62d} we deduce that 
\begin{align} & \label{:63e} 
\supN \sups \sup _{x\not=w\in\Sr }   
\frac
{|\widetilde{\Psi }^{\n }_{rs} (x,\mathsf{y}) -  \widetilde{\Psi }^{\n }_{rs} (w ,\mathsf{y}) |}
{| x -  w |} 
\\ \notag \le &
 \{ \supN \sups |\mmmm |  \}  
 + 
\cref{;6x} 
\sum_{\ell = 2}^{\ellell -1} \{ 
 \sups |\9 | \} 
+ 
\cref{;6z}\{ \sum_{\y _i \in \Sri } 
\frac{1}{| y _i | ^{\ellell } - \br ^{\ellell } }\} 
.\end{align}
Combining this with \eqref{:21f}, and \eqref{:63x}--\eqref{:63z}, we deduce that 
\begin{align*}
\Hrk & 
\supset \left\{ 
\bigcap _{\ell = 1}^{\ellell -1}
\mathsf{U}_{r,\ell ,k/(\ellell \cref{;6x})}\right\} 
\bigcap 
\bar{\mathsf{U}}_{r,\ellell ,k/(\ellell \cref{;6z})}
.\end{align*}
Hence, we obtain 
\begin{align}\label{:63g}&
\muN (\Hrk ^c) \le  
\left\{\sum_{\ell = 1}^{\ellell -1}  
\muN (\mathsf{U}_{r,\ell ,k/(\ellell \cref{;6x})} ^c)
\right\}
+ 
\muN (\bar{\mathsf{U}}_{r,\ellell ,k/(\ellell \cref{;6z})}^c)
.\end{align}
This together with \thetag{H.5} implies \eqref{:21g}, which completes the proof. 
\end{proof}

\section{Proof of \tref{l:22}} \label{s:7}
In this section, we complete the proof of  \tref{l:22}. 
For this we check the conditions of \thetag{H.5}. 
We begin with \eqref{:63p}, the first condition of \thetag{H.5}. 
\begin{lem} \label{l:71}
Assume $  S = \mathbb{C}$ and \thetag{H.2}. 
Then \eqref{:63p} follows from \eqref{:22d}. 
\end{lem}
\begin{proof} 
Let $ \br $ be as in \eqref{:qg0}. We divide the set 
$ \Sri = \{ \br \le |x | < \infty \} $ in \eqref{:63z} into two parts 
$ \Rrr = \{ \br \le |x | < \brr \} $ and $ \Rrrr = \{ \brr \le |x | < \infty \} $. 
Let 
\begin{align*}& 
\mathsf{V}_{1,k}  = 
\{ \xx \in \mathsf{S} \, ;\, \{ 
\sum_{\x _i \in \Rrr  } \frac{1}{| x _i |^{\ellell }- \br ^{\ellell } } \} 
  \le \frac{k}{2}\}   
\\ \notag &
\mathsf{V}_{2,k}  = 
\{ \xx \in \mathsf{S} \, ;\, \{ 
\sum_{\x _i \in \Rrrr  } \frac{1}{| x _i |^{\ellell }- \br ^{\ellell } } \}  
  \le \frac{k}{2}\},  
  \quad \text{ where }  \xx =\sum_{i}\delta_{x _i}
. \end{align*}
Then clearly 
$ \bar{\mathsf{U}}_{r,\ellell ,k} \supset \mathsf{V}_{1,k}\cap \mathsf{V}_{2,k}$. 
To estimate $ \mathsf{V}_{1,k} $, we observe that for $ \xx =\sum_{i}\delta_{\x _i }$
\begin{align}& \notag %%\label{:71e}
\sum_{\x _i \in \Rrr }
\frac{1}{| x _i |^{\ellell }- \br ^{\ellell } } \le 
\{ 
\sup_{\x _i \in \Rrr }
\frac{1}{| x _i |^{\ellell }- \br ^{\ellell } } \}\cdot \xx (\Rrr )
.\end{align}
Here $ \xx (\Rrr )$ is the number of points $ \x _i $ in $ \Rrr $. 
Considering this, we set  
\begin{align*}&
\mathsf{V}_{3,k} =
\{ \xx \in \mathsf{S} \, ; \,  
\sup_{\x _i \in \Rrr }
\frac{1}{| x _i |^{\ellell }- \br ^{\ellell } } \le \sqrt{{\ak }/{2}}
\}, \quad 
\\ \notag &
\mathsf{V}_{4,k} = 
\{ \xx \in \mathsf{S} ;\,  \xx (\Rrr ) \le \sqrt{{\ak }/{2}} \}. 
\end{align*}
Then we have $ \mathsf{V}_{1,k} \supset \mathsf{V}_{3,k} \bigcap \mathsf{V}_{4,k}  $. 
We therefore obtain 
$\bar{\mathsf{U}}_{r,\ellell ,k} \supset \mathsf{V}_{2,k}\cap \mathsf{V}_{3,k}\cap \mathsf{V}_{4,k}$ 
by combining these two inclusions. Hence, we deduce \eqref{:63p} from 
\begin{align}\label{:71g}&
\limi{k} \limsup_{ \n \to \infty } \muN (\mathsf{V}_{l,k}^c) = 0 
\quad \text{ for all }l=2,3,4 
.\end{align}
We will check \eqref{:71g} for each $ l =2,3,4 $. 

As for \eqref{:71g} with $ l =2 $, according to the Chebyshev inequality, we have  
\begin{align}\label{:71h}
\muN (\mathsf{V}_{2,k}^c) &\le 
\frac{2}{k} E^{\muN }[
 \sum_{\x _i \in \Rrrr  } \frac{1}{| x _i |^{\ellell }- \br ^{\ellell }}  ] 
\\ \notag &
= 	\frac{2}{k} \int_{ \Rrrr } \{ \frac{1}{ | x  |^{\ellell }- 
 \br ^{\ellell }} \} \rNone (x)	 dx  
\\ \notag &  	= 
\frac{2}{k} 
\int_{ \Rrrr } \{ \frac{{| x  |^{\ellell }}}
{ | x  |^{\ellell } -  \br ^{\ellell }} 
\frac{1}{{| x  |^{\ellell }}} \}
\rNone (x)	 dx  
\\ \notag &
\le \frac{2}{k} 
\{ 
\frac{ \brr ^{\ellell }}
{ \brr ^{\ellell }-  \br ^{\ellell }}
\}
 \cdot  \int_{\Rrrr }
  \{ \frac{1 }{| x  |^{\ellell }} \}
  \rNone (x )   dx  
.\end{align}
Here we used the fact that 
$ {t ^{\ellell }}/\{t ^{\ellell }- \br ^{\ellell }\}$ 
is decreasing in $ t \in (\br ,\infty )$, 
which implies 
\begin{align*}&
\sup_{x  \in \Rrrr }
\frac
 {| x  |^{\ellell }}
 {| x  |^{\ellell }-  \br ^{\ellell }}
 \le  \frac
{ \brr ^{\ellell }}
{ \brr ^{\ellell }-  \br ^{\ellell }}
\end{align*}
By \eqref{:22d} and \eqref{:71h}, we obtain 
 \eqref{:71g} with $ l=2 $.

We next consider \eqref{:71g} with $ l =3 $. Let 
\begin{align*}&
 U_{k}  = 
\{ x \in \Rrr ; \br ^{\ellell }  \le |x | ^{\ellell } < \br ^{\ellell } + {\sqrt{2/\ak }}\} 
.\end{align*}
It is not difficult to see that $  U_{k} $ is non-increasing and 
$ \limi{k} U_{k}  = \emptyset $. We note that 
\begin{align}\label{:71j}
\mathsf{V}_{3,k}^c &= \{ \xx \in \mathsf{S} ; 
\inf_{\x _i \in \Rrr }\{ 	{| x _i |^{\ellell }- \br ^{\ellell } } \} <  \sqrt{2/\ak }  \} 
\\ \notag &  = 
\left\{ \xx \in \mathsf{S} ; \ 1 \le \xx ( U_{k}  ) \right\} 
.\end{align}
Here we use the convention such that $ \inf \emptyset = \infty $; 
that is, we interpret $ \xx \not\in \mathsf{V}_{3,k}^c  $ when $ \xx (\Rrr )=0$. 
Let $ \cref{;71k}= \sup  \{ \rNone (x ); \n \in  \mathbb{N}   , x \in \Rrr \} $. 
Then by  \eqref{:21b}, we have $ \Ct \label{;71k} < \infty $. 
From the second equality in \eqref{:71j} and the Chebyshev inequality, we obtain 
\begin{align}\label{:71k}&
\muN (\mathsf{V}_{3,k} ^c) 
 \le  E^{\muN }[\xx   ( U_{k}  )] 
 = \int_{ U_{k}  } \rNone (x )dx 
 \le   \cref{;71k}  \int_{ U_{k}  }dx  
.\end{align}
Hence, we deduce \eqref{:71g} with $ l =3 $ from \eqref{:71k} and $ \limi{k} U_{k} = \emptyset $. 

We finally consider \eqref{:71g} with $ l =4 $. 
From the Chebyshev inequality, we obtain 
\begin{align}&\notag %%\label{:71l}&
\muN (\mathsf{V}_{4,k} ^c) \le  \sqrt{\frac{2}{\ak }}   E^{\muN }[\xx   (\Rrr )] = 
  \sqrt{\frac{2}{\ak }}    \int_{\Rrr } \rNone (x )dx  
\le   \sqrt{\frac{2}{\ak }}   \cref{;71k}   \int_{\Rrr }   dx  
.\end{align}
This immediately yields \eqref{:71g} with $ l =4 $. 
\end{proof}

We proceed with \eqref{:63q}, the second condition of \thetag{H.5}.

\begin{lem} \label{l:72}
Let the same assumptions as \lref{l:71} hold. 
Then  \eqref{:63q} follows from \eqref{:22e} and \eqref{:22f}. 
\end{lem}

\begin{proof}
By \eqref{:22f}, we can and do choose $ \{\br \}$ and 
$ \Ct \label{;61} > 0 $ in such a way that 
\begin{align} \label{:72c}&
\supN \| \supP |\vellbri ^{\mathsf{p}} | \|_{\LmNone } \le \cref{;61}3^{- r }
\quad \text{ for all } r \in  \mathbb{N} ,\ 1\le \ell < \ellell 
.\end{align}
We note that 
$ \mathsf{v}_{\ell ,\br \bs } (\xx ) = \vellbri (\xx )- \vellbsi (\xx )$. 
Then by \eqref{:63y}, we see that 
\begin{align} \label{:72d}
\muN (\{ \mathsf{U}_{r,\ell ,k} \}^c) = & 
	\muN ( \sups \supP | \vellbri ^{\mathsf{p}} - \vellbsi ^{\mathsf{p}}|> k )
\\ \notag \le & 
\muN (  \supP | \vellbri ^{\mathsf{p}}|> \frac{k}{2}) + 
	\muN ( \sups \supP | \vellbsi ^{\mathsf{p}}|> \frac{k}{2})
\\ \notag \le & 
\muN (  \supP | \vellbri ^{\mathsf{p}}|> \frac{k}{2}) +
	\sum_{s = r+1}^{\infty} \muN (  \supP |\vellbsi ^{\mathsf{p}}|> \frac{k}{2})
\\ \notag \le &
 \frac{2}{k}\cdot \{ \sum_{s = r}^{\infty} \| 
 \supP |\vellbsi ^{\mathsf{p}} | \|_{\LmNone }\}
.\end{align}
Here we used the Chebyshev inequality in the last line. 
By \eqref{:72c} and \eqref{:72d}, we have 
\begin{align*}&
\supN \muN (\{\mathsf{U}_{r,\ell ,k}\}^c) 
\le 
\frac{2}{k}\cdot \frac{\cref{;61} 3^{-r}}{1-3^{-1}}
.\end{align*}
Hence, $ \limi{k}\supN \muN (\{\mathsf{U}_{r,\ell ,k}\}^c) =0$, 
which implies \eqref{:63q}. 
\end{proof}

\noindent 
{\em Proof of \tref{l:22}. } 
From \lref{l:71} and \lref{l:71}, we deduce that 
the assumption \thetag{H.5} in \tref{l:63} holds. 
Hence from \tref{l:63}, we obtain \thetag{H.3}. 
Therefore \tref{l:22} follows from \tref{l:21}.   
\qed



%% References


\small{

}


\begin{thebibliography}{999}
\DN\ttl[1]{\textit{#1},}

\bibitem{ark}   Albeverio,~S.\ {\it et al.} \ttl{Analysis and geometry on configuration spaces: the Gibbsian case} Joul. Func. Anal, {\bf 157}, (1998) 242-291.  



\bibitem{agz} Anderson, Greg W., {\it et al}, \ttl{An introduction to random matrices} Cambridge studies in advanced mathematics {\bf 118} (2010). 

\bibitem{for} Forrester, P.J., \textit{Log gases and random matrices} Princeton University Press (2010).


\bibitem{Fr} Fritz,~J. \ttl{Gradient dynamics of infinite point systems}  Ann. Prob. {\bf 15 } (1987) 478-514. 

\bibitem{johansson.03} Johansson, K., \ttl{Discrete polynuclear growth and determinantal processes} Commun. Math. Phys. {\bf 242} (2003) 277-329.

\bibitem{johansson.11} Johansson, K., \ttl{Non-colliding Brownian motions and the extended Tacnode processes} (preprint).

\bibitem{knt.04} Katori, M., Nagao, T.\ and Tanemura, H., \ttl{Infinite systems of non-colliding Brownian particles} Adv. Stud. in Pure Math. {\bf 39} “Stochastic Analysis on Large Scale Interacting Systems”, pp.283-306, (Mathematical Society of Japan, Tokyo, 2004); arXiv:math.PR/0301143.

\bibitem{kt.07-ptrf} Katori, M.\ and Tanemura, H., \ttl{Infinite systems of non-colliding generalized meanders and Riemann-Liouville differintegrals} Probab. Th. Rel. Fields, {\bf 138} (2007) 113-156.

\bibitem{kt.07} Katori,~M., Tanemura,~H., \ttl{Noncolliding Brownian motion and determinantal processes}  J. Stat. Phys. {\bf  129 } (2007) 1233-1277.  

\bibitem{La1} Lang,~R., \ttl{Unendlich-dimensionale Wienerprocesse mit Wechselwirkung I}  Z. Wahrschverw. Gebiete  {\bf  38  } (1977) 55-72   

\bibitem{La2} Lang,~R., \ttl{Unendlich-dimensionale Wienerprocesse mit Wechselwirkung II} Z. Wahrschverw. Gebiete  {\bf  39  } (1977) 277-299.   


\bibitem{mehta} Mehta, M., \ttl{Radom matrices} (Third Edition) Elsevier 2004.  

\bibitem{o.dfa} Osada, H., \ttl{Dirichlet form approach to infinitely dimensional Wiener processes with singular interactions} Commun.\ Math.\ Physic.\ (1996), 117-131.
%

\bibitem{o.m} Osada,~H., \ttl{Interacting Brownian motions with measurable potentials} Proc. Japan Acad. Ser. A Math. Sci. {\bf 74} (1998), { no. 1}, 10--12. 

\bibitem{o.p} Osada,~H. \ttl{Positivity of the self-diffusion matrix of interacting Brownian particles with hard core} Probab. Theory Relat. Fields {\bf 112} (1998), 53--90. 

\bibitem{o.col} Osada,~H., \ttl{Non-collision and collision properties of Dyson's model in infinite dimension and other stochastic dynamics whose equilibrium states are determinantal random point fields} in Stochastic Analysis on Large Scale Interacting Systems, eds.\ T.\ Funaki and H.\ Osada, Advanced Studies in Pure Mathematics \textbf{39}, 2004, 325-343. 

\bibitem{o.tp} Osada,~H., \ttl{Tagged particle processes and their non-explosion criteria} J. Math. Soc. Japan, {\bf 62}, No.\ {\bf 3} (2010), 867-894. 

\bibitem{o.isde} Osada,~H., \ttl{Infinite-dimensional stochastic differential equations related to random matrices} Probab. Theory Relat. Fields {\bf 153} (2012), 471--509.  


\bibitem{o.rm} Osada,~H., \ttl{Interacting Brownian motions in infinite dimensions with logarithmic interaction potentials} (to appear in Annals of Probability). 


 \bibitem{p-spohn.02} Pr\"{a}hofer M.\ and Spohn H., \ttl{Scale invariance of the PNG droplet and the Airy process} J. Stat.
Phys. {\bf 108} (2002) 1071-1106.

\bibitem{ruelle2}  Ruelle,~D., \ttl{Superstable interactions in classical statistical mechanics} Commun. Math. Phys. {\bf 18} (1970) 127--159. 


\bibitem{shiga} Shiga,~T.\ \ttl{A remark on infinite-dimensional Wiener processes with interactions} \rm Z.\ Wahrschverw. Gebiete  {\bf 47} (1979) 299-304   


\bibitem{sp.2} Spohn, H., \ttl{Interacting Brownian particles: a study of Dyson's model}  {\bf } In: Hydrodynamic Behavior and Interacting Particle Systems, ed. by G.C. Papanicolaou, IMA Volumes in Mathematics  {\bf 9 },  Springer-Verlag (1987) 151-179.

\bibitem{T2} Tanemura,~H., \ttl{A system of infinitely many mutually reflecting Brownian balls in $\mathbb{R}^d $}  Probab.\  Theory Relat.\  Fields {\bf  104} (1996) 399-426. 

\bibitem{tane.udf} Tanemura,~H., \ttl{Uniqueness of Dirichlet forms associated with systems of infinitely many Brownian balls in $ \mathbb{R}^d $}  Probab.\  Theory Relat.\  Fields {\bf 109} (1997) 275-299. 


\bibitem{yuu.05} Yoo, H.\ J., \ttl{Dirichlet forms and diffusion processes for Fermion random point fields} J.\ Functional Analysis {\bf 219} (2005) 143-160. 

\bibitem{yoshida} Yoshida,~M.W. \!\! \ttl{Construction of infinite-dimensional interacting diffusion processes through Dirichlet forms} Probab. \! Theory Relat. \! Fields {\bf 106} (1996) 265-297. 

\end{thebibliography}
\end{document}